 \RequirePackage{anyfontsize} 
\documentclass[sn-mathphys]{sn-jnl}
\usepackage{amsmath}
\usepackage{amsfonts}
\usepackage{amssymb}
\usepackage{physics}
\usepackage{subfig}
\usepackage{abrege}

\theoremstyle{thmstyleone}%
\newtheorem{theorem}{Theorem}
\newtheorem{proposition}{Proposition}

\theoremstyle{thmstyletwo}%
\newtheorem{example}{Example}%
\newtheorem{remark}{Remark}%

\theoremstyle{thmstylethree}%
\newtheorem{definition}{Definition}%

\usepackage{algorithm} 
\usepackage{algpseudocode}

\DeclareMathOperator{\inertia}{\textbf{inertia}}
\DeclareMathOperator*{\argmin}{arg\; min}     
\DeclareMathOperator{\diag}{diag}         
\newcommand{\inv}[1][1]{^{-#1}}
\renewcommand{\iff}{\Leftrightarrow}

\begin{document}
	
\title[Newton-type methods for nonconvex-nonconcave minmax optimization]{Newton and interior-point methods for (constrained) nonconvex-nonconcave minmax optimization with stability and instability guarantees}

\author*[1]{\fnm{Raphael} \sur{Chinchilla}}\email{raphaelchinchilla@ucsb.edu}

\author[2]{\fnm{Guosong} \sur{Yang}}\email{guosong.yang@rutgers.edu}

\author[1]{\fnm{Jo\~{a}o} \sur{P.~Hespanha }}\email{hespanha@ucsb.edu}

\affil[1]{\orgdiv{Center for Control, Dynamical Systems, and Computation}, \orgname{University of California, Santa Barbara}, \orgaddress{\postcode{93106}, \state{California}, \country{United States of America}}}

\affil[2]{\orgdiv{Department of Electrical and Computer Engineering}, \orgname{Rutgers University–New Brunswick}, \orgaddress{\postcode{08854}, \state{New Jersey}, \country{United States of America}}}

%
\abstract{We address the problem of finding a local solution to a nonconvex-nonconcave minmax optimization using Newton type methods, including primal-dual interior-point ones. The first step in our approach is to analyze the local convergence properties of Newton's method in nonconvex minimization. It is well established that Newton's method iterations are attracted to any point with a zero gradient, irrespective of it being a local minimum. From a dynamical system standpoint, this occurs because every point for which the gradient is zero is a locally asymptotically stable equilibrium point. We show that by adding a multiple of the identity such that the Hessian matrix is always positive definite, we can ensure that every non-local-minimum equilibrium point becomes unstable (meaning that the iterations are no longer attracted to such points), while local minima remain locally asymptotically stable. Building on this foundation, we develop Newton-type algorithms for minmax optimization, conceptualized as a sequence of local quadratic approximations for the minmax problem. Using a local quadratic approximation serves as a surrogate for guiding the modified Newton's method towards a solution. For these local quadratic approximations to be well-defined, it is necessary to modify the Hessian matrix by adding a diagonal matrix. We demonstrate that, for an appropriate choice of this diagonal matrix, we can guarantee the instability of every non-local-minmax equilibrium point while maintaining stability for local minmax points. Using numerical examples, we illustrate the importance of guaranteeing the instability property. While our results are about local convergence, the numerical examples also indicate that our algorithm enjoys good global convergence properties.}


\maketitle

\section{Introduction}

In minmax optimization, one minimizes a cost function which is itself obtained from the maximization of an objective function. Minmax optimization is a powerful modeling framework, generally used to guarantee robustness to an adversarial parameter such as accounting for disturbances in model predictive control \cite{bemporad_robust_1999,copp_simultaneous_2017}, security related problems \cite{pita2008deployed,yang2021adaptive}, or training neural networks to be robust to adversarial attacks \cite{madry_towards_2019}. It can also be used as a framework to model more general problem such as sampling from unknown distributions using generative adversarial networks \cite{goodfellow_generative_2014}, reformulating stochastic programming as minmax optimization \cite{chinchilla_optimization-based_2019,ChinchillaHespanhaDec20,bandi_tractable_2012}, or producing robustness of a stochastic program with respect to the probability distribution \cite{rahimian_distributionally_2019}. Minmax optimization is also known as minimax or robust optimization.  Minmax optimization is related to bilevel optimization \cite{dempe2002foundations,colson_overview_2007,beck_brief_2023}, as minmax optimization can sometimes be used to find solutions to bi-level optimization when the inner and outer maximization have antisymmetric criteria.

Finding a global minmax point for nonconvex-nonconcave problems is generally difficult, and one has to settle for finding a local minmax point. Surprisingly, only recently a first definition of unconstrained local minmax was proposed in \cite{jin_what_2019}, and the definition of constrained local minmax in \cite{dai_optimality_2020}. 

In optimization, Newton's method consists of applying Newton's root finding algorithm to obtain a point for which the gradient is equal to zero. In convex minimization, the only such points are (global) minima \cite[Theorem 2.5]{nocedal_numerical_2006}. Likewise, in convex-concave minmax optimization (meaning that the function is convex in the minimizing variable and concave in the maximizing variable), Van Neumans's Theorem \cite{v1928theorie} states that the min and the max commute, which implies that the only points for which the gradient is zero are solutions to the optimization. This means that both in convex minimization and convex-concave minmax optimization, using Newton's root finding method to obtain a point for which the gradient is zero is a good strategy to solve the optimization problem.
	
In contrast, for nonconvex minimization or nonconvex-nonconcave minmax optimization, the gradient can be zero at a point even if such point is not a solution to the optimization. So using Newton's root finding method to obtain a point for which the gradient is equal to zero is not a good strategy to find a (local) solution to the optimization. The foundation of our work involves examining Newton's method iterations through the lens of dynamical systems. By analyzing the linearization of the dynamics, we deduce that every equilibrium point (i.e., a point with a zero gradient) is locally asymptotically stable, which is why the iterations of the Newton's method are attracted to them. \textbf{The key contribution of this article} is to study how to modify the Newton's method such that it is only attracted to (local) solutions of the optimization, and repelled by any equilibrium points that are not (local) solutions.

Our paper's initial contribution is an examination of the local convergence properties of a modified Newton's method for minimization in which a multiple of the identity matrix is added to the Hessian such that the resulting matrix is positive definite \cite[Chapter 3.4 ``Newton's method with Hessian modification'']{nocedal_numerical_2006}. This modified Newton has two crucial properties. First, it can be shown to be equivalent to a sequence of local quadratic approximations to the minimization problem. Second, we demonstrate that incorporating this additive matrix renders every non-local-minimum equilibrium point unstable while maintaining stability for local minima. This simple modification ensures that the modified Newton's method has the property we refer to in the previous paragraph that the iterations are only attracted to equilibrium points that are local minima, and repelled by other equilibrium points. Utilizing analogous techniques, we establish similar results for primal-dual interior-point methods in constrained minimization. These findings (outlined in Section \ref{sc:minimization}) directly inspire the development of new Newton-type algorithms for minmax optimization.

 Drawing inspiration from the Newton's method for minimization, we develop Newton-type algorithms for minmax optimization, conceptualized as a series of local quadratic approximations of the minmax problem. For convex-concave functions, this quadratic approximation is just the second-order Taylor expansion, which leads to the (unmodified) Newton's method, accompanied by its well-established local convergence properties. However, for nonconvex-nonconcave functions, it is necessary to add scaled identity matrices to ensure that the local approximations possess finite minmax solutions (without mandating convex-concavity). Additive terms meeting this criterion are said to satisfy the Local Quadratic Approximation Condition (LQAC). Employing a sequence of local quadratic approximations acts as a surrogate for guiding the modified Newton's method towards a solution at each step. Nevertheless, we demonstrate that, unlike minimization, local quadratic approximation-based modifications are not enough to ensure that the algorithm can only converge towards local minmax points. Our minmax findings reveal that additional conditions are required on the modification to unsure the algorithm's convergence to an equilibrium point is guaranteed only if that point is a local minmax. To streamline the presentation, we first introduce this result in Section \ref{sc:minmax-unconstrained} for unconstrained minmax, then expand it to primal-dual interior-point methods for constrained minmax in Section \ref{sc:minmax-constrained}.

The conditions described above to establish the equivalence between local minmax and local asymptotic stability of the equilibria to a Newton-type iteration are directly used to construct a numerical algorithm to find local minmax. By construction, when this algorithm converges to an equilibrium point, its is guaranteed to obtain a local minmax. One could be tempted to think that the issue of getting instability
for the equilibria that are not local minima (in Theorems \ref{th:stab-newton-minimization} and \ref{th:stab-ip-minimization})
or that are not local minmax (in Theorems \ref{th:stab-newton-minmax} and \ref{th:stab-ip-minmax}) is just a
mathematical curiosity, which in practice makes little
difference. However, our numerical examples in sections \ref{sc:alg-development-numerical-examples-benchmark} and \ref{sc:alg-development-numerical-examples-MPC} show otherwise. Most
especially the pursuit-evasion MPC problem in section \ref{sc:alg-development-numerical-examples-MPC}, where finding a local minmax (rather than an equilibrium that is not local minmax) leads to a completely different control. Specifically, if the instability property is not guaranteed, the evader is not able escape from the pursuer. It is important to emphasize that our results fall shy of guaranteeing \emph{global} asymptotic convergence to a local minmax, as the algorithm could simply never converge. However, our numerical examples also show that our algorithm seems to enjoy good global convergence properties in practice. Using the results of this paper, we have created a solver for minmax optimization and included it in the solvers of \texttt{TensCalc}\footnote{\url{https://github.com/hespanha/tenscalc}} \cite{hespanha_tenscalc_2022}; this solver was used to generate the numerical results  we present.

\paragraph{Notation:} The set of real numbers is denoted by $\eR$. Given a vector $v\in \eR^n$, its transpose is denoted by $v'$. The operation $\diag(v)$ creates a matrix with diagonal elements $v$ and off-diagonal elements $0$. The matrix $I$ is the identity, $\oneb$ is the matrix of ones and $\zerob$ the matrix of zeros; their sizes will be provided as subscripts whenever it is not clear from context. 
If a matrix $A$ only has real eigenvalues, we denote by $\lambda_{min}(A)$ and $\lambda_{max}(A)$ its smallest and largest eigenvalues.  The inertia of $A$ is denoted by $\inertia(A)$, and is a 3-tuple with the number of positive, negative and zero eigenvalues of $A$.

Consider a differentiable function $f:\eR^n\times \eR^m \mapsto \eR^p$. The Jacobian (or gradient if $p=1$) at a point $(\bar x,\bar y)$ according to the $x$ variable is a matrix of size $n\times p$ and is denoted by $\grad_xf(\bar x,\bar y)$, and analogously for the variable $y$. When $p=1$ and $f(\cdot)$ is twice differentiable, we use the notation $\grad_{yx}f(\bar x,\bar y):=\grad_y\big(\grad_xf\big)(\bar x,\bar y)$ which has sizes $m\times n$. We use analogous definition for $\grad_{xy}f(\bar x,\bar y)$, $\grad_{xx}f(\bar x,\bar y)$ and $\grad_{yy}f(\bar x,\bar y)$.

\subsection{Literature Review}

Traditionally, robust optimization focused on the convex-concave case, with three main methods. The first type of method is based on Von Neuman's minmax theorem \cite{v1928theorie} that states that the min and the max commute when the problem is convex-concave and the optimization sets are convex and compact. Solving the minmax then simplifies to finding a point that satisfies the first order condition. While there are many different methods to achieve this, many of them can be summarized by the problem of finding the zeros of a monotone operator \cite{ryu_primer_2016}. The second type of methods consists on reformulating the minmax as a minimization problem which has the same solution as the original problem. This is generally done using either robust reformulation through duality theory or tractable variational inequalities \cite{ben-tal_robust_2002,ben-tal_robust_2009,bertsimas_theory_2011,colson_overview_2007}. The third, cutting-set methods, solves a sequence of minimization where the constraint of each minimization is based on subdividing the inner maximization \cite{mutapcic_cutting-set_2009}.

Motivated by some of the shortcomings of these methods and the necessities of machine learning, research on minmax optimization started to study first-order methods based on variations of gradient descent-ascent. The results tend to focus on providing convergence complexity given different convexity/~concavity assumptions on the target function. We can divide these first order methods in three families. The first familie solves the minmax by (approximately) solving the maximization each time the value of the minimizer is updated. When this is done using first order methods, it is generally referred to as multi-step gradient descent ascent, unrolled gradient descent ascent or GDmax, and the minimizer is updated by a single gradient descent whereas the maximizer is updated by several gradient ascent steps. A second family uses single step, where the minimizer and maximizer are updated at each iteration. For both of these two first families, the gradient iterations can include variations such as using different step sizes for the minimization and maximization, or using momentum. A third family, which is completely different from what is described for other ones, is to include the gradient from different time steps in the computation, such as the past one (as in optimistic gradient descent-ascent), the midpoint between the current and future points (as in extra gradient descent-ascent) and at future point (as in proximal point). The literature on first-order methods is very extensive, and we refer to  \cite{jin_what_2019,nouiehed_solving_nodate,metz_unrolled_2017,mokhtari_unified_2020, lin_near-optimal_2020,liu_first-order_2020, nemirovski_prox-method_2004,fiez_gradient_2020,mertikopoulos_optimistic_2019} and the references within for the exposition on some of these methods and their convergence properties.

In recent years, researchers have also started to work on algorithms that use second order derivatives to determine the directions. These algorithm, in their major part have not attracted as much attention as first order methods. In the Learning with Opponent Learning Awareness (LOLA), the minimizer anticipates the play of the maximizer using the Jacobian of the maximizer's gradient \cite{foerster_learning_2018,letcher_stable_2021}. In competitive gradient descent, both minimizer and maximizer use the cross derivative of the Hessian to compute their direction \cite{schafer_competitive_2020}. In follow the ridge, the gradient ascent step is corrected by a term that avoids a drift away from local maxima \cite{wang_solving_2019}. In the total gradient descent-ascent, similarly to LOLA, the descent direction is computed by taking to total derivative of a function which anticipates the maximizer's response to the minimizer \cite{fiez_convergence_2019}. Finally, the complete Newton borrows ideas from follow the ridge and total gradient to obtain a Newton method which prioritizes steps towards local minmax \cite{zhang_newton-type_2020}. These three last algorithms are shown to only converge towards local minmax under some conditions, but in none of them it is addressed the issue of how to adjust the Hessian far away from a local minmax point.

Recently, some second order methods have been proposed for the nonconvex-strongly-concave case, where the minimizer update is a descent direction of the objective function at its maximum. They either use cubic regularization \cite{luo_finding_2021,chen_escaping_2021} or randomly perturb the Hessian \cite{huang_efficiently_2022}. Because of some of the assumptions these work make, most important the strong-concavity of the objective function with respect to the maximizer, they are able to establish complexity analysis and guarantee. It is also worth mention that these algorithms are all multi-step based, meaning they (approximately) solve the maximization between each update of the minimizer, whereas our algorithm updates both the minimizer and the maximizer simultaneously.

\section{Minimization} \label{sc:minimization}
 
Let $f:\Xcal \to \eR$ be a twice continuously differentiable cost function defined in a set $\Xcal\subset \eR^{n_x}$ where $n_x$ is a positive integer \footnote{The subscript $_x$ is used to indicate that $n_x$ refers to the size of the variable $x$. We introduce this notation now in anticipation of Section \ref{sc:minmax} where we have both minimization and maximization variables.}, and consider the minimization problem
\begin{equation} \label{eq:minimization}
\min_{x\in \Xcal} f(x).
\end{equation}
We recall that a point $x^*$ is called a local minimum of $f(\cdot)$ if there exist $\delta>0$ such that $f(x^*)\le f(x)$ for all $x\in\{x\in\Xcal:\norm{x-x^*}<\delta\}$. We will study the property of Newton type algorithms to solve \eqref{eq:minimization} in two distinct cases, when $\Xcal=\eR^{n_x}$ and when $\Xcal$ is defined by equality and inequality constraints.

\subsection{Unconstrained minimization} \label{sc:minimization-unconstrained}

Let $\Xcal=\eR^{n_x}$, which is referred to as unconstrained minimization in the literature, in which case \eqref{eq:minimization} simplifies to 
\begin{equation} \label{eq:minimization_unco}
	\min_{x\in \eR^{n_x}} f(x).
\end{equation}
If $f(\cdot)$ is twice continuously differentiable in a neighborhood of a point $x$ and $\grad_xf(x)=\zerob$ and $\grad_{xx}f(x)\succ0$, then $x$ is a local minimum of $f(\cdot)$ \cite[Chapter 2]{nocedal_numerical_2006}.

An extremely popular method to solve a minimization problem is to use Newton's root finding method to obtain a point $x$ such that $\grad_xf(x)=\zerob$. In its most basic form, the algorithm's iterations are given by
\begin{equation}\label{eq:Newton-min}
	x^+=x+d_x=x-\grad_{xx}f(x)\inv \grad_x f(x).
\end{equation}
where we use the notation $x^+$ to designate the value of $x$ at the next iteration. Newton's method biggest advantage is that it converges very fast near any point that satisfies the first order condition $\grad_xf(x)=\zerob$: at least linearly but possibly superlinearly when the function is Lipschitz \cite[Theorem 3.6]{nocedal_numerical_2006}. However, this is also precisely Newton's method biggest limitation for nonconvex minimization, because it does not distinguish a local minimum from any other point satisfying the first order condition. Let us further illustrate this limitation with an example.
\begin{example} \label{example1}
	Consider the optimization,
	\begin{align}\label{eq:example-u3}
		\min_{x\in\eR} x^3-3x,
	\end{align}
	for which $\forall x\in\eR$,
	\begin{align*}
		&f(x):= x^3-3x, &
		&\grad_x f(x)=3x^2-3, &
		&\grad_{xx}f(x)=6x.
	\end{align*}
	The corresponding Newton iteration \eqref{eq:Newton-min} is of the form
	\begin{align*}
		x^+=x-\frac{3x^2-3}{6x},
	\end{align*}
	for which both the local minimum $x^{\min}:= 1$ and the local
	maximum $x^{\max}:= -1$ are locally asymptotically stable
	equilibria with superlinear convergence. Specifically, 
	\begin{align*}
		\begin{cases}
			& x_0>0 \Rightarrow x_k \to x^{\min}:= 1,  \text{(local minimum)},\\
			&  x_0<0 \Rightarrow x_k \to x^{\max}:= -1, \text{(local maximum)},\\
			& x_0=0 \Rightarrow \text{iteration fails since $\grad_{xx}f(x)=6x$ is not invertible.}
		\end{cases}
	\end{align*}	
	Moreover, the iteration never actually ``converges'' to the global
	``infimum'' $x\to-\infty$.  
\end{example}

In order to address this limitation, a widely used modification of Newton's method for unconstrained nonconvex optimization \cite[Chapter 3.4]{nocedal_numerical_2006}, is obtained by modifying the basic Newton method such that $d_x$ is obtained from solving the following local quadratic approximation to \eqref{eq:minimization}
\begin{align}
d_x&=\argmin_{\bar d_x} f(x)+\grad_xf(x)'\bar d_x + \frac{1}{2}\bar d_x \grad_{xx}f(x) \bar d_x +\frac{\epsilon_x(x)}{2} \norm{d_x}^2 \label{eq:minlqac}
\\&=\argmin_{\bar d_x} f(x)+\grad_xf(x)'\bar d_x + \frac{1}{2}\bar d_x (\grad_{xx}f(x)+\epsilon_x(x)I) \bar d_x \nonumber
\\&=-(\grad_{xx}f(x)+\epsilon_x(x)I)\inv \grad_x f(x) \nonumber
\end{align}
with $\epsilon_x(x)\ge0$ chosen such that $(\grad_{xx}f(x)+\epsilon_x(x)I)$ is positive definite. For twice differentiable strongly-convex functions we can choose $\epsilon_x(x)=0$ and this corresponds to the classical Newton's method. However, when $f(\cdot)$ is not strongly-convex, the minimization in \eqref{eq:minlqac} is only well-defined if  $\grad_{xx}f(x)+\epsilon_x(x)I$ is positive definite, which requires selecting a strictly positive value for $\epsilon_x(x)$, leading to a modified Newton's method. Regardless of whether $f(\cdot)$ is convex, the positive definiteness of $\grad_{xx}f(x)+\epsilon_x(x)I$ guarantees that $d_x ' \grad_x f(x)=-\grad_x f(x)(\grad_{xx}f(x)+\epsilon_x(x)I)\inv\grad_x f(x)<0$  and therefore $d_x$ is a descent direction at $x$ \cite{nocedal_numerical_2006}. The corresponding Newton iteration to obtain a local minimum is then given by
\begin{equation}\label{eq:newtonmin}
x^+=x+d_x=x-(\grad_{xx}f(x)+\epsilon_x(x)I)\inv \grad_x f(x).
\end{equation}
Let us analyze how this modification impacts the convergence in our previous example.

\setcounter{example}{0}
\begin{example}[Continuation]
	For the optimization in~\eqref{eq:example-u3}, the modified
	Newton step in~\eqref{eq:newtonmin} becomes
	$x^+=x-\frac{3x^2-3}{6x+\epsilon_x(x)}$
	with $\epsilon_x(\cdot)$ such that
	\begin{align}\label{eq:epsilon-example-u3}
		\begin{cases}
			\epsilon_x(x) \ge 0 & x>0,\\
			\epsilon_x(x) > -6x & x\le 0.
		\end{cases}
	\end{align}
	In this case, 
	\begin{align*}
		\begin{cases}
			x_0>x^{\max}:= -1 \Rightarrow &x_k \to x^{\min}:= 1 \text{(local minimum)},\\
			x_0<x^{\max}:= -1 \Rightarrow &x_k \to -\infty \text{ (global ``infimum'')},\\
			x_0=x^{\max}:= -1 \Rightarrow &x_k=x^{\max}, \forall k  \text{(unstable equilibrium)}.
		\end{cases}
	\end{align*}
	Selecting the function $\epsilon_x(\cdot)$ with $\epsilon_x(\cdot) = 0$ around $x^{\min}$ results in superlinear convergence to $x^{\min}$, but if $\epsilon_x(\cdot) > 0$, the convergence is only linear. For example, picking $\epsilon_x(x)=-6x + \eta$   with $\eta>0$, \eqref{eq:epsilon-example-u3} holds for all $x$, but the
	modified Newton step in~\eqref{eq:newtonmin} becomes
	$x^+=x-\frac{3x^2-3}{\eta}$, which is just a gradient descent.
\end{example}

\bigskip

The following result generalizes the conclusion from the previous example by establishing that the positive definiteness of $\grad_{xx}f(x)+\epsilon_x(x)I$ not only guarantees that $d_x$ is a descent direction, but also that every locally asymptotically stable (LAS) equilibrium point of the Newton iteration \eqref{eq:newtonmin} is a local minimum. 

\begin{theorem}[Stability and instability of modified Newton method for unconstrained minimization] \label{th:stab-newton-minimization}
	Let $x$ be an equilibrium point in the sense that $\grad_{x}f(x)=\zerob$. Assume that $\grad_{xx}f(x)$ is invertible and that  $\grad_{xx}f(\cdot)$ is differentiable in a neighborhood around $x$. Then for any function $\epsilon_x(\cdot)$ that is constant in a neighborhood around $x$ and satisfies $\grad_{xx}f(x)+\epsilon_x(x)I\succ \zerob$ one has that if:
	\begin{enumerate}
		\item[i)] $x$ is a local minimum of \eqref{eq:minimization_unco}, then it is a LAS equilibrium of \eqref{eq:newtonmin}.
		\item[ii)] $x$ is not a local minimum of \eqref{eq:minimization_unco}, then it is an unstable equilibrium of \eqref{eq:newtonmin}.
	\end{enumerate} 
\end{theorem}
The theorem's first implication is that if the modified Newton iteration starts sufficiently close to a strict local minimum, it will converge at least linearly fast to it. One could think that it would always be preferable to have $\epsilon_x(x)=0$ if $\grad_{xx}f(x)\succ0$, in which case not only stability can be trivially obtained but also that the Newton method has superlinear convergence if $f(\cdot)$ is Lipschitz \cite[Theorem 3.6]{nocedal_numerical_2006}. However, in practice, there are situations for which one might want to take $\epsilon_x(x)>0$. A typical case happens if the smallest eigenvalue of $\grad_{xx}f(x)$ is positive but very small, which might bring numerical issues when computing the Newton step $\grad_{xx}f(x)\inv\grad_xf(x)$. This issue can be fixed by taking $\epsilon_x(x)>0$, and Theorem \ref{th:stab-newton-minimization} guarantees that doing so will not impair (at least locally) the algorithm's capacity to converge towards a local minimum.

The theorem's second implication is, in a way, even more relevant than the first one. As we mentioned earlier, the regular Newton's method (meaning, with $\epsilon_x(x)=0$) is infamously known to be attracted to any point that satisfies $\grad_{x}f(x)=\zerob$, regardless of whether it is a local minimum, a saddle point, or a local maximum. What Theorem \ref{th:stab-newton-minimization} is essentially saying is that the modified Newton is only attracted to local minima, and that any other equilibrium point repels the iteration. In essence, this means that the modified Newton's method cannot converge towards a point that is not a local minimum, thus fixing one of the biggest drawbacks of the regular Newton's method.

While it goes beyond the point of this article, notice that for large values of $\epsilon_x(x)$,
	$$
		x^+=x-(\grad_{xx}f(x)+\epsilon_x(x) I)\inv\grad_{x}f(x)\approx x- \epsilon_x(x)\inv\grad_{x}f(x)
$$
	which shows that the modified Newton's step \eqref{eq:newtonmin} essentially
	becomes a gradient descent step with a small step size
	$\epsilon_x(x)\inv$. This also shows that, by keeping
	$\epsilon_x(x)\inv$ sufficiently large, the iteration \eqref{eq:newtonmin} could be
	made descent with respect to the cost. However, this would
	be achieved at the cost of losing superlinear
	convergence.

\begin{proof}[Proof of Theorem \ref{th:stab-newton-minimization}]
	From our assumption that $\grad_{xx}f(x)$ is invertible, $x$ is a local minimum if and only if $\grad_{xx}f(x)\succ0$. This comes from the second order necessary condition for minimization \cite[Chapter 2]{nocedal_numerical_2006}.

	Let us now prove the stability and instability properties. The first step in our analysis is to calculate the Jacobian of $(\grad_{xx} f(x)+\epsilon_x(x)I)\inv \grad_x f(x)$ that appears in \eqref{eq:newtonmin} at an equilibrium point $x$. Using the differentiability of $\grad_{xx}f(\cdot)$ and that $\epsilon_x(\cdot)$ is constant in a neighborhood of $x$, we obtain that
	\begin{multline*}
	\grad_{x}\Big((\grad_{xx} f(x)+\epsilon_x(x)I)\inv \grad_{x} f(x)\Big)=(\grad_{xx} f(x)+\epsilon_x(x)I)\inv \grad_{xx} f(x) + \\\sum_{i=1}^N  \grad_{x}[(\grad_{xx} f(x)+\epsilon_x(x)I)\inv]_i\grad_{x} f(x)^{(i)}
	\end{multline*}
	where  $\grad_{x} f(x)^{(i)}$ is the i\textsuperscript{th} element of $\grad_{x} f(x)$ and $[(\grad_{xx} f(x)+\epsilon_x(x)I)\inv]_i$ is the i\textsuperscript{th} column of $(\grad_{xx} f(x)+\epsilon_x(x)I)\inv$. Since $(\grad_{xx} f(x)+\epsilon_x(x)I)$ is positive definite, $\grad_x[(\grad_{xx} f(x)+\epsilon_x(x)I)\inv]_i$ is well defined and since $x$ is an equilibrium point, $\grad_{x} f(x)^{(i)}=0$ for $i\in\{1\dots N\}$ and therefore the Jacobian of right-hand side of \eqref{eq:newtonmin} is given by
	\begin{equation}\label{eq:jacob_min}
	\grad_{x}\Big(x-(\grad_{xx} f(x)+\epsilon_x(x)I)\inv \grad_{x} f(x)\Big)=I-(\grad_{xx} f(x)+\epsilon_x(x)I)\inv \grad_{xx} f(x). 
	\end{equation}

	The main argument of the proof is based on the following result. Let $v$ be an eigenvector associated to an eigenvalue $\rho$ of \eqref{eq:jacob_min}. Then
	\begin{align}
	\Big(I-(\grad_{xx} f(x)+\epsilon_x(x)I)\inv \grad_{xx}f(x)\Big)v&=\rho v \nonumber \\
	\iff (1 - \rho) v &= (\grad_{xx} f(x)+\epsilon_x(x)I)\inv \grad_{xx}f(x) v \nonumber \\
	\iff \Big(\rho \grad_{xx}f(x) + (\rho - 1) \epsilon_x(x)I\Big) v &= \zerob \label{eq:eig_jacob_min}
	\end{align}
	Therefore, $ \rho $ is an eigenvalue of \eqref{eq:jacob_min} if and only if $ \rho \grad_{xx}f(x) + (\rho - 1) \epsilon_x(x)I$ is singular.
	
	We remind the reader that given a dynamical system, if the system's dynamic equation is continuously differentiable, a point is a LAS equilibrium point if all the eigenvalues of the linearized system are inside the unit circle. Conversely, if at least one of the eigenvalues of the linearized system is outside the unit circle, then the system is unstable \cite[Chapter 8]{hespanha_linear_2018}.

	From \eqref{eq:eig_jacob_min}, $\rho=0$ is an eigenvalue if and only if $\epsilon_x(x)=0$, which, by construction, can only happen if $x$ is a local minimum, in which case $x$ is a LAS equilibrium point of \eqref{eq:newtonmin}, as expected. 
	
	For $\rho\neq0$, let us rewrite this expression as $\grad_{xx}f(x) + \mu \epsilon_x(x)I$ with $ \mu := 1 - 1/\rho$. We conclude that $x$ is a LAS equilibrium point of \eqref{eq:newtonmin} if $\grad_{xx}f(x) + \mu \epsilon_x(x)$ is nonsingular $\forall \mu \in [0, 2]$. Conversely, $x$ is an unstable equilibrium point of \eqref{eq:newtonmin} if $\grad_{xx}f(x) + \mu \epsilon_x(x)$ is singular for some $\mu \in [0, 2]$.
	
	If $x$ is a local minimum, then $\lambda_{min}(\grad_{xx}f(x))>0$. As $\epsilon_x(x)>0$, we conclude that  $\lambda_{min}(\grad_{xx}f(x) + \mu \epsilon_x(x)I)>0$  for every $\mu\ge0$ and therefore $x$ is a LAS equilibrium point of \eqref{eq:newtonmin}. Conversely, if $x$ is not a local minimum then $\lambda_{min}(\grad_{xx}f(x))<0$. By construction of $\epsilon_x(x)$, we have that $\lambda_{min}(\grad_{xx}f(x) + \mu \epsilon_x(x)I)>0$, which, by continuity of the eigenvalue, implies $\exists \mu \in (0,1)$ such that $\lambda_{min}(\grad_{xx}f(x) + \mu \epsilon_x(x)I)=0$. Therefore $x$ is an unstable equilibrium point of \eqref{eq:newtonmin}. 
\end{proof}

\subsection{Constrained minimization} \label{sc:minimization-constrained}

Our results from the previous section can also be extended to consider the case with more general constraint with the minimization set $\Xcal$ involving equality and inequality constraints of the form $$\Xcal=\{x\in \eR^n: G_x(x)=\zerob, F_x(x)\le \zerob\}$$ where the functions $G_x:\eR^{n_x}\to \eR^{l_x}$ and $F_x:\eR^{n_x}\to\eR^{m_x}$ are all twice continuously differentiable \footnote{Similar to $n_x$, the subscript $_x$ is used to indicate that the functions $G_x(\cdot)$ and $F_x(\cdot)$ are associated to the minimization variable $x$. We introduce this notation now in anticipation of Section \ref{sc:minmax} where we have both minimization and maximization variables.}.
It will be convenient for the development of the primal-dual interior-point method to use slack variables and rewrite \eqref{eq:minimization} as
\begin{equation}\label{eq:minimization_slack}
\min_{x,s_x:G_x(x)=\zerob,F_x(x)+s_x=\zerob,s_x\ge \zerob} f(x).
\end{equation}
where $s_x\in\eR^{m_x}$.

Similar to what we have in the unconstrained minimization, we want a second order conditions to determine whether a point is a local minimum. Consider the function 
$$L(z)=f(x)+\nu_x' G_x(x) + \lambda_x'(F_x(x)+s_x),$$
where we use the shorthand notation $z:=(x,s_x,\nu_x,\lambda_x)$. $L(z)$ is essentially the Lagrangian of \eqref{eq:minimization_slack}. In order to present the second order conditions, we need to define two concepts, the linear independence constraint qualification and strict complementarity \cite[Definitions 12.4 and 12.5]{nocedal_numerical_2006}.
\begin{definition}[LICQ and strict complementarity]
	Let the set of active inequality constraints for the minimization be defined by
	\begin{equation*}
	\Acal_x(x)=\{i=1,\dots,m_x : F_x^{(i)}(x)=0\}
	\end{equation*}
	where $F_x^{(i)}(x)$ denote the i\textsuperscript{th} element of $F_x(x)$. Then:
	\begin{itemize}
		\item  The linear independence constraint qualification (LICQ) is said to hold at $z$ if the vectors in the set
		\begin{equation*}
		\{\grad_x G_x^{(i)}(x), i=1,\dots,l_x\}\bigcup\{\grad_x F_x^{(i)}(x), i \in \Acal_x(x) \}
		\end{equation*}
		are linearly independent.
		\item Strict complementarity is said to hold at $x$ if $\lambda_x^{(i)}>0 \ \forall i\in \Acal_x(x)$ 
	\end{itemize}
\end{definition}

We have almost all the ingredients to present the second order condition for constrained minimization. For unconstrained minimization, a sufficient condition for a point $x$ to be a local minimum is that $\grad_{x}f(x)=0$ and $\grad_{xx}f(x)\succ0$. If it were not for the inequality constraints in \eqref{eq:minimization_slack}, we would be able to state the second order conditions using gradients and Hessians of $L(z)$. The inequality constraints make the statement a bit more complicated. The role of the gradient will be played by
\begin{equation}\label{eq:kkt_min}
g(z,b):=\begin{bmatrix}
\grad_{x} L(z) \\
\lambda_x \odot s_x - b \oneb \\
G_x(x)  \\
F_x(x) +s_x
\end{bmatrix}
\end{equation}
with $\odot$ denoting the element wise Hadamard product of two vectors and $b\ge0$ the barrier parameter (its role will be explained shortly). The role of $\grad_{xx}f(x)$ in the unconstrained minimization will be played by the matrix
\begin{equation}
H_{zz}f(z)=
\begin{bmatrix}
\grad_{xx} L(z) & 0 & \grad_{x} G_x(x) & \grad_{x}F_x(x) \\
0 & \diag(\lambda_x) & 0 & \diag(s_x^{1/2}) \\
\grad_{x} G_x(x)'& 0 & 0 & 0\\
\grad_{x}F_x(x)'& \diag(s_x^{1/2}) &0 & 0  \\
\end{bmatrix}.
\end{equation}
We also remind the reader that the inertia $\inertia(A)$ of a symmetric matrix $A$ is a 3-tuple with the number of positive, negative and zero eigenvalues of $A$.

\begin{proposition}[Second order sufficient conditions for constrained minimization] \label{prop:constrsecondorder_min}
	Let $z$ be an equilibrium point in the sense that $g(z,0)=\zerob$ with $\lambda_x,s_x\ge\zerob$. If the LICQ and strict complementarity hold at $z$ and
	\begin{equation} \label{eq:second_order_inertia_constrained_min}
	\inertia( H_{zz}f(z))=(n_x+m_x,l_x+m_x,0)
	\end{equation}
	then $x$ is a local minimum of \eqref{eq:minimization_slack}. 
\end{proposition}
While this result is relatively well known, we present its proof in Appendix \ref{app:proof_propo_secon}. The proof also makes it easier to understand the proof of the second order sufficient conditions for constrained minmax optimization.

\subsubsection{Primal-dual interior-point method}

Let $d_z:=(d_x,d_s,d_\nu,d_\lambda)$ be the update direction for $z$, which will play an equivalent role to $d_x$ in the unconstrained case. A basic primal-dual interior-point method finds a candidate solution to  \eqref{eq:minimization_slack} using the iterations 
\begin{equation}\label{eq:basic_ip_min} 
z^+=z+\alpha d_z=z-\alpha \grad_{z}g(z,b)'\,\inv g(z,b)
\end{equation}
where the barrier parameter\footnote{ 
	The term ``barrier parameter'' comes from the connection between primal-dual and (log) barrier interior-point methods. This connection will become more clear bellow around \eqref{eq:minimization_slack_barier} as we deduce a local second order approximation of \eqref{eq:minimization_slack}. } $b$ is slowly decreased to $0$, so that $z$ converges to a root of $g(z,0)=\zerob$ while $\alpha\in (0,1]$ is chosen at each step such that the feasibility condition $\lambda_x,s_x>\zerob$ hold \cite[Chapter 19]{nocedal_numerical_2006}. This basic primal-dual interior-point has similar limitation as a (non-modified) Newton method for unconstrained minimization: it might converge towards an equilibrium point that is not a local minimum and $\grad_{z}g(z,b)$ might not be invertible. Similar to what we have done in the unconstrained case, we can modify this basic primal-dual interior-point method such that the update direction $d_z$ is obtained from a quadratic program that locally approximates \eqref{eq:minimization_slack}. The rest of this section will be spent mostly constructing such quadratic program.

Let us start with $\Xcal$ described only by equality constraints (\ie no $F_x(x)$ and no $s_x$), in which case $L(z)=f(x)+\nu_x' G_x(x)$.  Consider the optimization
\begin{align}
&\min_{\bar d_x:G_x(x)+\grad_xG_x(x)'\bar d_x=\zerob} L(z)+\bar d_x'\grad_xL(z)+\frac{1}{2} \bar d_x'(\grad_{xx}L(z)+\epsilon_x(z)I)\bar d_x \label{eq:min_equal_constr_QP}
\\&=\min_{\bar d_x:G_x(x)+\grad_xG_x(x)'\bar d_x=\zerob} L(z)+\bar d_x'\grad_xL(z)+\frac{1}{2} \bar d_x'\grad_{xx}L(z)\bar d_x + \frac{\epsilon_x(z)}{2} \norm{\bar d_x}^2, \nonumber
\end{align}
which locally approximates \eqref{eq:minimization_slack} around $(x,\nu_x)$
\footnote{
Notice that we use the second order linearization of the Lagrangian $L(z)$ as the cost function in \eqref{eq:min_equal_constr_QP}, not the one of $f(x)$. The justification is that, if $x^*$ is a local minimum of \eqref{eq:minimization_slack} with associated Lagrange multiplier $\nu^*$, then $x^*$ is also a local minimum of 
\begin{equation*}
\min_{x:G_x(x)=\zerob} f(x) + \nu_x^*\,'G_x(x).
\end{equation*}
Evidently, $\nu_x^*$ is not know in advance, so instead one uses the value of $\nu_x$ at the current iteration, which leads to the local approximation \eqref{eq:min_equal_constr_QP}.}. 
If $\grad_xG_x(x)$ is full column rank, we can choose $\epsilon_x(z)$ large enough such that the solution of \eqref{eq:min_equal_constr_QP} is well defined and unique. To show that, let us look at \eqref{eq:min_equal_constr_QP} as an optimization in its own right. Let $\bar d_\nu$ be the Lagrange multiplier  and define the function $\bar g(\bar d_x,\bar d_\nu)$ which is the function $g(z,b)$ defined in \eqref{eq:kkt_min} but now for problem \eqref{eq:min_equal_constr_QP}:
\begin{equation}\label{eq:kkt_min_QP}
\bar g(\bar d_x,\bar d_\nu):=\begin{bmatrix}
& \grad_xL(z)+(\grad_{xx}L(z)+\epsilon_x(z)I)\bar d_x + \grad_xG_x(x)\bar d_\lambda  \\
& G_x(x)+\grad_xG_x(x)'\bar d_x 
\end{bmatrix}.
\end{equation}
So if one takes any $\epsilon_x(z)\ge0$ large enough such that
\begin{equation}
\inertia\qty(\begin{bmatrix}
\grad_{xx} L(z)+\epsilon_x(z) & \grad_{x} G_x(x) \\
\grad_{x} G_x(x)'& 0 \\
\end{bmatrix})=(n_x,l_x,0),
\end{equation}
then we guarantee that any point $\bar d_x,\bar d_\nu$ that satisfies $\bar g(\bar d_x,\bar d_\nu)=\zerob$ will be a strict local minimum of \eqref{eq:min_equal_constr_QP} (see Proposition \ref{prop:constrsecondorder_min}). Moreover, this choice of $\epsilon_x(z)$ also guarantees that \eqref{eq:min_equal_constr_QP} is a strongly convex quadratic optimization, which, with the fact that $\grad_xG_x(x)$ is full column rank, means that the solution $(\bar d_x,\bar d_\nu)$ is \emph{unique}. Therefore, we will take the update directions $(d_x,d_\nu)$  to be the solution  $(\bar d_x,\bar d_\nu)$. Moreover, with some algebra, one can show that the solution to \eqref{eq:min_equal_constr_QP} is given by
\begin{align*}
\begin{bmatrix}
d_x \\
d_\nu 
\end{bmatrix}&=-\begin{bmatrix}
\grad_{xx} L(z)+\epsilon_x(z) & \grad_{x} G_x(x) \\
\grad_{x} G_x(x)'& 0 \\
\end{bmatrix}\inv 
\begin{bmatrix}
\grad_xL(z)\\
G_x(x)
\end{bmatrix}
\\&=-(\grad_{z} g(x,b)'+\diag([\epsilon_x(z)\oneb_{n_x},\zerob_{l_x}]))\inv g(x,b).
\end{align*}

Let us now address the case in which there there are inequality constraints. The challenge is to take into account the constraint $s_x\ge0$. To address this, let us start by relaxing the inequality constraint from \eqref{eq:minimization_slack} and including it in the cost as the barrier function $-b\oneb'\log(s_x)$ (the $\log(\cdot)$ is element wise). 
\begin{equation}\label{eq:minimization_slack_barier}
\min_{x,s_x:G_x(x)=\zerob,F_x(x)+s_x=\zerob} f(x) -b\oneb'\log(s_x).
\end{equation}
This is a relaxation because $-b\oneb'\log(s_x)$ only accepts $s\ge0$ and goes to $+\infty$ if $s_x\to0$. The optimization \eqref{eq:minimization_slack_barier} only has equality constraints, so similar to what we did in \eqref{eq:min_equal_constr_QP}, let us construct a local second order approximation of \eqref{eq:minimization_slack_barier} around $z$:
\begin{multline}\label{eq:min_QP}
\min_{\substack{\bar d_x,\bar d_s:\\G_x(x)+\grad_xG_x(x)'\bar d_x=\zerob,\\F_x(x)+s_x+\grad_xF_x(x)'\bar d_x+\bar d_s=\zerob}} L(z)-b\oneb'\log(s_x)+\bar d_x'\grad_xL(z)+\bar d_s'(\lambda_x-b\oneb\oslash s_x)
\\+\frac{1}{2}\bar d_x'(\grad_{xx}L(z)+\epsilon_x(z)I)\bar d_x+\frac{1}{2}\bar d_s'\diag(\lambda_x\oslash s_x)\bar d_s
\end{multline}
where $\oslash$ designates the element wise division of two vectors. Equation \eqref{eq:min_QP} is not exactly a second order approximation because instead of using as quadratic term for $\bar d_s$ the matrix $b\diag(s_x)\inv[2]$ (which is the actual matrix given by second order approximation of $-b\oneb'\log(s_x+d_s)$ around $s_x$), we used the matrix $\diag(\lambda_x\oslash s_x)$. This is a relatively well known substitutions for interior-point methods, and is what makes it be a primal-dual interior-point method instead of a barrier interior-point method. The technical justification is that, if we were at a point such that $g(z,b)=\zerob$, the two would be equivalent as $\lambda_x \odot s_x - b \oneb=\zerob$. In practice, it has been observed that this modified linearization tends to perform better because it provides directions $d_s$ that also take into account the current value of $\lambda_x$ in the quadratic form, which helps to get a direction $d_z$ that does no violate the constraints $\lambda_x,s_x>0$ \cite[Chapter 19.3]{nocedal_numerical_2006}. 
 
Because \eqref{eq:min_QP} is a quadratic program with linear equality constraints, just as it was the case for \eqref{eq:min_equal_constr_QP}, we can use the exact same reasoning to choose  $\epsilon_x(z)$. Let us define the matrices
 \begin{equation}
J_{zz}f(z)=
\begin{bmatrix}
\grad_{xx} L(z) & 0 & \grad_{x} G_x(x) & \grad_{x}F_x(x) \\
0 & \diag(\lambda_x\oslash s_x) & 0 & I \\
\grad_{x} G_x(x)'& 0 & 0 & 0\\
\grad_{x}F_x(x)'& I &0 & 0  \\
\end{bmatrix}
\end{equation}
and $E(z):=\diag(\epsilon_x(z)\oneb_{n_x},\zerob_{m_x+l_x+m_x})$. If $\epsilon_x(z)$ is chosen large enough such that $\inertia(J_{zz}+E(z))=(n_x+m_x,l_x+m_x,0)$, then the solution $(\bar d_x,\bar d_s)$ of \eqref{eq:min_QP} and associated Lagrange multipliers $(\bar d_\nu,\bar d_\lambda)$ are unique. With some algebra, one could show that the solution of \eqref{eq:min_QP} is
\begin{align*}
d_z & = -(J_{zz}f(z)+E(z))\inv S\inv g(z,b)
\\&= -(\grad_{z}g(z,b)'+E(z))\inv g(z,b)
\end{align*}
where $S:=\diag(\oneb_{n_x},s_x,\oneb_{l_x+m_x})$. Putting it all together, the modified primal-dual interior-point is governed by the equation
\begin{equation}\label{eq:ip_min} 
z^+=z+\alpha d_z=z-\alpha (\grad_{z}g(z,b)'+E(z))\inv g(z,b),
\end{equation}
where $\alpha\in(0,1]$ is chosen such that $\lambda_x,s_x>\zerob$. Conveniently, because we used $\diag(\lambda_x\oslash s_x)$ for the second order linearization of the barrier, when $\epsilon_x(x)=0$, we recover the basic primal-dual interior-point method from \eqref{eq:basic_ip_min}. We refer to \cite[Chapter 19]{nocedal_numerical_2006} for a complete description of an algorithm using \eqref{eq:ip_min}, including a strategy to decrease the barrier parameter $b$. Alternatively, we describe such strategy in Section \ref{sc:alg-development-numerical-examples} for the minmax optimization case. 

We can now state a result connecting the stability/instability of any equilibrium point of the modified primal-dual interior-point method to such point being or not a local minimum. The theorem says essentially the same thing as Theorem \ref{th:stab-newton-minimization}: On the one hand, even if $\inertia(J_{zz}f(z))=(n_x+m_x,l_x+m_x,0)$, taking $\epsilon_x(z)>0$ will not impair the algorithm's capacity to converge towards a local minimum; this can be useful, for instance, if $\inertia(J_{zz}f(z))$ has an eigenvalue close to $0$. On the other hand, using the modified primal-dual interior-point method essentially guarantees that the algorithm can only converge towards an equilibrium point if such point is a local minimum, thus fixing the issue of primal-dual interior-point methods being attracted to any equilibrium point, regardless of whether such point is a local minimum.

\begin{theorem}[Stability and instability of modified primal-dual interior-point method for constrained minimization] \label{th:stab-ip-minimization}
	Let $\alpha=1$ and $(z,b)$ with $b>0$, be an equilibrium point in the sense that $g(z,b)=\zerob$. Assume the LICQ and strict complementarity hold at $z$, that $J_{zz}f(z)$ is invertible, and that $J_{zz}f(\cdot)$ is differentiable on a neighborhood around $z$. Then for any function $\epsilon_x(\cdot)$ that is constant in a neighborhood around z and satisfies $\inertia(J_{zz}+E(z))=(n_x+m_x,l_x+m_x,0)$ one has that if:
	\begin{enumerate}
		\item[i)] $z$ is a local minimum of \eqref{eq:minimization_slack}, then it is a LAS equilibrium of \eqref{eq:ip_min}.
		\item[ii)] $z$ is not a local minimum of \eqref{eq:minimization_slack}, then it is an unstable equilibrium of \eqref{eq:ip_min}.
	\end{enumerate} 
\end{theorem}

\begin{proof}[Proof sketch]
First, using the same arguments as in the proof of Theorem \ref{th:stab-newton-minimization}, we conclude that the Jacobian of the dynamic system \eqref{eq:ip_min} around a point $z$ for which $g(z,b)=\zerob$ is
\begin{equation}
I- \alpha \Big(J_{zz}f(z)+E(z)\Big)\inv S\inv \grad_zg(z,b)'=I- \alpha \Big(J_{zz}f(z)+E(z)\Big)\inv J_{zz}f(z)
\end{equation}

Second, it is straightforward to check that $H_{zz}f(z)=S^{1/2}J_{zz}f(z)S^{1/2}$ which, using Sylvester's law of inertia \cite[Theorem 1.5]{zhang_schur_2005}, means that $\inertia(H_{zz}f(z))=\inertia(J_{zz}f(z))$. This means that one can check the second order conditions in \eqref{eq:second_order_inertia_constrained_min} by using $J_{zz}f(z)$. 

Let us define the matrix
\begin{equation*}
	R(\mu)=Z_x(z)'\begin{bmatrix}
		\grad_{xx} L(z)+\mu \epsilon_x(z)I & 0 \\
		0 & \diag(\lambda_x\oslash s_x) \\
	\end{bmatrix}Z_x(z)
\end{equation*}
where $Z_{x}(z)\in\eR^{n_x+m_x,n_x-l_x}$ is a  matrix with full column rank such that 
\begin{equation}
	\begin{bmatrix}
		\grad_{x} G_x(x)' & \zerob \\
		\grad_{x}F_x(x)' & I
	\end{bmatrix}Z_x(z)=\zerob.
\end{equation}
Using the same arguments as in the proof of Proposition \ref{prop:constrsecondorder_min}, we conclude that
$$
\inertia(J_{zz}f(z)+E(z))=\inertia(R(\mu))+(l_x+m_x,l_x+m_x),
$$
which implies that $\inertia(J_{zz}f(z)+E(z))=(n_x+m_x,l_x+m_x)$ is equivalent to $R(1)\succ0$ and that the second order sufficient condition is equivalent to $R(0)\succ0$. This means that the rest of the theorem's proof is analogous to the one of Theorem \ref{th:stab-newton-minimization}, but instead of looking at the sign of the smallest eigenvalue of $\grad_{xx}f(x)+\mu \epsilon_x(z)I$, one looks at the sign of the smallest eigenvalue of the matrix $R(\mu)$.

If $z$ is a local minimum, then $\lambda_{min}(R(0))>0$. As $\epsilon_x(z)\ge0$, we conclude that $\lambda_{min}(R(\mu))>0$ for every $\mu\ge0$ and therefore $z$ is a LAS equilibrium
point of \eqref{eq:basic_ip_min}.

Conversely, if $z$ is not a local minimum, $\lambda_{min}(R(0))<0$. By construction, $\epsilon_x(z)$ is such that $\lambda_{min}(R(1))>0$, therefore, by continuity of the eigenvalue, there is a $\mu\in(0,1)$ such that $\lambda_{min}(R(\mu))=0$ and therefore $z$ is an unstable equilibrium point of \eqref{eq:basic_ip_min}.	
\end{proof}


\section{Minmax optimization}\label{sc:minmax}

Consider the minmax optimization problem
\begin{equation} \label{eq:minmaxconstrained}
\min_{x\in\Xcal}\max_{y\in\Ycal(x)}f(x,y)
\end{equation}
where $f:\eR^{n_x}\times\eR^{n_y}\to \eR$ is a twice continuously differentiable objective function,  $\Xcal\subset \eR^{n_x}$ is the feasible set for $x$ and $\Ycal:\Xcal\rightrightarrows\eR^{n_y}$ is a set-valued map that defines an $x$ dependent feasible set for $y$; we do not make any convexity or concavity assumption on $f(\cdot)$, $\Xcal$ and $\Ycal(\cdot)$. We chose $\Ycal(\cdot)$ to be dependent on $x$ because this describes the most general application. Moreover, having the constraints of the inner maximization to depend on the value of outer maximization is often necessary in problems such as robust Model Predictive Control or in bi-level optimization. Furthermore, notice that we do not make any assumption on whether the min and the max commute (and this would not be well defined as $\Ycal(\cdot)$ depends on $x$). A solution $(x^*,y^*)$ to \eqref{eq:minmaxconstrained} is called a global minmax and satisfies
\begin{align*}
f(x^*,y)\le f(x^*,y^*)\le \max_{ \tilde y \in \Ycal(x)}f(x, \tilde y) \qquad \forall (x,y)\in \Xcal\times\Ycal(x^*).
\end{align*}
We will look at two representations of $\Xcal \tand \Ycal(\cdot)$: first when $\Xcal=\eR^{n_x}$ and $\Ycal=\eR^{n_y}$, which is known in the literature as the unconstrained case; second a more general representation in which $\Xcal$ and $\Ycal(\cdot)$ are defined using equality and inequality constraints.

A point $(x^*,y^*)$ is said to be a local minmax of \eqref{eq:minmaxconstrained} if there exist a constant $\delta_0>0$ and a positive function $h(\cdot)$ satisfying $h(\delta)\to 0$ as $\delta\to0$, such that for every $\delta \in (0,\delta_0]$ and for every $(x,y)\in$  $\{x\in\Xcal:\norm{x-x^*}\le\delta \}$ $\times \{y\in\Ycal(x^*):\norm{y-y^*}\le h(\delta)\}$ we have
\begin{align*}
f(x^*,y)\le f(x^*,y^*)\le \max_{\tilde y \in \Ycal(x): \norm{ \tilde y-y^*}\le h(\delta)} f(x, \tilde y)
\end{align*}
\cite{jin_what_2019,dai_optimality_2020}. Inspired by the properties of the modified Newton and primal-dual interior-point methods for minimization in Section \ref{sc:minimization}, we want to develop a Newton-type iterative algorithm of the form
\begin{equation}\label{eq:iterative}
\begin{bmatrix}
x^+\\
y^+
\end{bmatrix}=
\begin{bmatrix}
x\\
y
\end{bmatrix}+
\begin{bmatrix}
d_{x}\\
d_{y}
\end{bmatrix}.
\end{equation}
where $d_x$ and $d_y$ satisfy the following properties:
\begin{enumerate}
	\item[P1:] At each time step, $(d_x,d_y)$ is obtained from the solution of a quadratic program that locally approximates \eqref{eq:minmaxconstrained} and therefore $(x^+,y^+)$ can be seen as an improvement over $(x,y)$. This acts as a surrogate for guiding the modified Newton's method towards a solution at each step.
	\item[P2:] The iterations of \eqref{eq:iterative} can converge towards an equilibrium point only if such point is a local minmax. Similar to what was the case in minimization (see Example \ref{example1}), a pure Newton method will be attracted to any equilibrium point. This makes sure that the iterations will not be attracted to equilibrium points that are not local minmax.
	\item[P3:] The iterations of \eqref{eq:iterative} can converge to any local minmax. This property means that any modification to Newton's method needs to keep local minmax as attractor.
\end{enumerate}

\subsection{Unconstrained minmax} \label{sc:minmax-unconstrained}

We start by considering the case where $\Xcal=\eR^{n_x}$ and $\Ycal(\cdot)=\eR^{n_y}$ such that \eqref{eq:minmaxconstrained} simplifies to
\begin{equation} \label{eq:minmax}
\min_{x\in \eR^{n_x}}\max_{y\in \eR^{n_y}} f(x,y).
\end{equation}
For this case, \cite{jin_what_2019} establishes second order sufficient conditions to determine if a point $(x,y)$ is a local minmax which can be stated in terms of the inertia of the matrix
\begin{equation*}
\grad_{zz}f(x,y) :=\begin{bmatrix}
\grad_{xx}f(x,y) & \grad_{xy}f(x,y) \\
\grad_{yx}f(x,y) & \grad_{yy}f(x,y)
\end{bmatrix}.
\end{equation*}
We recall that the inertia $\inertia(A)$ of a symmetric matrix $A$ is a 3-tuple with the number of positive, negative and zero eigenvalues of $A$.

\begin{proposition}[Second order sufficient condition for unconstrained minmax]\label{prop:second-order-unconstrained-minmax}
	Let $(x,y)$ be an equilibrium point in the sense that $\grad_x f(x,y)=0$ and $\grad_y f(x,y)=0$. If 
	\begin{equation} \label{eq:secondorderuncont}
	\inertia(\grad_{yy} f(x,y))=(0,n_y,0) \tand 
	\inertia(\grad_{zz}f(x,y))=(n_x,n_y,0)
	\end{equation}
	then $(x,y)$ is a local minmax. 
\end{proposition}
The second order conditions in \cite{jin_what_2019} are:
\begin{gather*}
\inertia(\grad_{yy} f(x,y))=(0,n_y,0)  \tand 
\\ \inertia(\grad_{xx} f(x,y) - \grad_{xy} f(x,y)\grad_{yy} f(x,y)\inv\grad_{yx} f(x,y))=(n_x,0,0),
\end{gather*}
which turn out to be equivalent to the inertia conditions in Proposition \ref{prop:second-order-unconstrained-minmax} in view of Haynsworth inertia additivity formula \cite[Theorem 1.6]{zhang_schur_2005}. Notice that the second order sufficient conditions are \emph{not} symmetric. A point might be a local minmax even if $\grad_{xx} f(x,y) \nsucc 0$ as long as $- \grad_{xy} f(x,y)\grad_{yy} f(x,y)\inv\grad_{yx} f(x,y)$ (which is positive) is large enough. So the second order conditions are what allow one to distinguish between an equilibrium point being a local minmax and a minmin, maxmax or maxmin. One can interpret the second order sufficient conditions as saying that $y\mapsto f(x,y)$ is strongly concave in a neighborhood around $(x,y)$ and $x\mapsto \max_{\tilde y:\norm{y-\tilde y}<\delta} f(x,\tilde y)$ is strongly convex in a neighborhood around $(x,y)$ for some $\delta>0$. Notice that these are only local properties around local minmax, as $f(\cdot)$ may be nonconvex-nonconcave away from local minmax points.

In order to obtain property P1, we propose to obtain the Newton direction $(d_x,d_y)$ for \eqref{eq:iterative} by solving the following local quadratic approximation to \eqref{eq:minmax}
\begin{multline}\label{eq:lqa}
\min_{\bar d_x}\max_{\bar d_y}f(x,y)+ \grad_{x}f(x,y)'\bar d_x+\grad_{y}f(x,y)'\bar d_y+\bar d_x'\grad_{xy}f(x,y)\bar d_y\\+\frac{1}{2}\bar d_x'\Big(\grad_{xx}f(x,y)+\epsilon_x(x,y)I\Big)\bar d_x  +\frac{1}{2}\bar d_y'\Big(\grad_{yy}f(x,y)-\epsilon_y(x,y)I\Big)\bar d_y
\end{multline}
with $\epsilon_x(\cdot)$ and $\epsilon_y(\cdot)$ chosen so that the minmax problem in \eqref{eq:lqa} has a unique solution, which means that the inner (quadratic) maximization must be strictly concave and that the outer (quadratic) minimization of the maximized function must be strictly convex, which turns out to be precisely the second order sufficient conditions in Proposition \ref{prop:second-order-unconstrained-minmax}, applied to the approximation in \eqref{eq:lqa}, which can be explicitly written as follows:
\begin{equation}\label{eq:lqac}
\tag{LQAC}
\begin{split}
\inertia\Big(\grad_{yy}f(x,y) - \epsilon_y(x,y)I\Big)&=(0,n_y,0) \tand \\
\inertia\Big(\grad_{zz}f(x,y)+E(x,y)\Big)&=(n_x,n_y,0)
\end{split}
\end{equation}
where $E(x,y) = \diag(\epsilon_x(x,y) \oneb_{n_x}, -\epsilon_y(x,y) \oneb_{n_y})$. We call these condition the Local Quadratic Approximation Condition (LQAC). It is straightforward to show that the Newton iterations \eqref{eq:iterative} with $(d_x,d_y)$ obtained from the solution to \eqref{eq:lqa} is given by
\begin{equation}\label{eq:newton}
\begin{bmatrix}
x^+\\
y^+
\end{bmatrix}=\begin{bmatrix}
x\\
y
\end{bmatrix}+
\begin{bmatrix}
d_x\\
d_y
\end{bmatrix}=
\begin{bmatrix}
x\\
y
\end{bmatrix}-
\Big(\grad_{zz}f(x,y)+E(x,y)\Big)\inv
\begin{bmatrix}
\grad_{x}f(x,y)\\
\grad_{y}f(x,y)
\end{bmatrix}.
\end{equation}

To obtain properties P2 and P3, we need all locally asymptotically stable equilibrium points of \eqref{eq:lqa} to be local minmax of \eqref{eq:minmax} and that all other equilibrium points of \eqref{eq:lqa} to be unstable. For the unconstrained minimization in Section \ref{sc:minimization-unconstrained}, to obtain the equivalent of properties P2 and P3 it was sufficient to simply select $\epsilon_x(\cdot)$ such that the local quadratic approximation \eqref{eq:minlqac} has a well-defined minimum (Theorem \ref{th:stab-newton-minimization}). However, for minmax optimization the \eqref{eq:lqac} does not suffice to guarantee that P2 and P3 hold. Our first counter example bellow show how the \eqref{eq:lqac} are not enough to ensure that P2 holds; our second counter example show how they are not enough to guarantee that P3 holds.

\begin{example}
	Consider $f(x,y)= 1.5x^2 -4xy +y^2$ for which the unique equilibrium point $x=y=0$ is not a local minmax point. Take $\epsilon_y(0,0)=4$ and $\epsilon_x(0,0)=0$ which satisfy \eqref{eq:lqac}. The Jacobian of the dynamics is
	$$
	I-\bigg(
	\begin{bmatrix}
	3& -4 \\
	-4 & 2
	\end{bmatrix} +
	\begin{bmatrix}
	0& 0 \\
	0 & -4
	\end{bmatrix} 
	\bigg)\inv		
	\begin{bmatrix}
	3& -4 \\
	-4 & 2
	\end{bmatrix}
	 \approx 	\begin{bmatrix}
	0& 0.72 \\
	0 & 0.54
	\end{bmatrix} 
	$$
	which has eigenvalues approximately equal to $(0,0.54)$. Therefore $(0,0)$ is a LAS equilibrium point of \eqref{eq:newton} even though it is not a local minmax point. 
\end{example}

\begin{example} Consider $ f(x, y) := -0.25 x^2 + xy - 0.5 y^2 $, for which the unique equilibrium point $x=y=0$ is a local minmax point. Take $\epsilon_y(0,0)=3$ and $\epsilon_x(0,0)=0.2$ which satisfy \eqref{eq:lqac}. The Jacobian of the dynamics is
\begin{equation*}
 I -\bigg(
 \begin{bmatrix}
 -0.5 & 1 \\
 1 & -1
 \end{bmatrix}  +
 \begin{bmatrix}
 0.3 & 0 \\
 0 & -3
 \end{bmatrix} 
 \bigg)\inv
 \begin{bmatrix}
-0.5 & 1 \\
1 & -1
\end{bmatrix} = \begin{bmatrix}
6 & -15 \\
1.5 & -3
\end{bmatrix},
\end{equation*}
for which the eigenvalues are $1.5\pm1.5i$. Therefore $(0,0)$ is an unstable equilibrium point of \eqref{eq:newton} even though it is a local minmax point. 
\end{example}
The main contribution of this section is a set of sufficient conditions that, in addition to \eqref{eq:lqac}, guarantee P2 and P3 hold.

\begin{theorem}[Stability and instability of modified Newton's method for unconstrained minmax] \label{th:stab-newton-minmax}
	Let $(x,y)$ be an equilibrium point in the sense that $\grad_{x}f(x,y)=\zerob$ and $\grad_{y}f(x,y)=\zerob$. Assume that $\grad_{zz}f(x,y)$ and $\grad_{yy}f(x,y)$ are invertible and that $\grad_{zz}f(\cdot)$ is differentiable on a neighborhood around $(x,y)$. Then there exist functions $\epsilon_x(\cdot)$ and $\epsilon_y(\cdot)$ that are constant in a neighborhood around $(x,y)$, satisfy the \eqref{eq:lqac} at $(x,y)$ and guarantee that if:
	\begin{enumerate}
		\item[i)] $(x,y)$ is a local minmax of \eqref{eq:minmax}, then it is a LAS equilibrium of \eqref{eq:newton}.
		\item[ii)] $(x,y)$ is not a local minmax of \eqref{eq:minmax}, then it is an unstable equilibrium of \eqref{eq:newton}.
	\end{enumerate}
\end{theorem}

The theorem's implications are similar to those of Theorem \ref{th:stab-newton-minimization}. On the one hand, if $(x,y)$ is a local minmax, then it is possible to construct functions $\epsilon_x(\cdot)$ and $\epsilon_y(\cdot)$ that guarantee that the modified Newton method can converge towards a local minmax. A natural choice for such function near a local minmax is to take $\epsilon_y(\cdot)=\epsilon_x(\cdot)=0$, which not only provides the stability result, but can also achieve superlinear convergence if $f(\cdot)$ is Lipschitz. On the other hand, if $(x,y)$ is an equilibrium point but not a local minmax, it is possible to construct functions $\epsilon_x(\cdot)$ and $\epsilon_y(\cdot)$ such that the algorithm's iterations \emph{cannot} converge towards it. This means that the modified Newton's method for minmax can only converge towards an equilibrium point if such point is a local minmax.

While the statement of Theorem \ref{th:stab-newton-minmax} is about existence, the proof is actually constructive. The functions $\epsilon_x(\cdot)$ and $\epsilon_y(\cdot)$ are not unique, and have to satisfy the following conditions:
\begin{enumerate}
	\item[i)] For the stability result, if $\epsilon_y(x,y)=0$, then the stability property is guaranteed by any $\epsilon_x(x,y)\ge0$. If $\epsilon_y(x,y)>0$, then $\epsilon_x(x,y)$ needs to be taken large enough to satisfy the condition in equation \eqref{eq:minmax-eps-cond} of the proof.
	\item[ii)] For the instability result:
	\begin{itemize}
		\item unless $\inertia(\grad_{yy}f(x,y))\ne(0,n_y,0)$ and 	 $\inertia(\grad_{zz}f(x,y))=(n_x,n_y,0)$, then it is sufficient for $\epsilon_x(x,y)$ and $\epsilon_y(x,y)$ to satisfy the \eqref{eq:lqac} to guarantee instability. 
		\item if $\inertia(\grad_{yy}f(x,y))\ne(0,n_y,0)$ and 	 $\inertia(\grad_{zz}f(x,y))=(n_x,n_y,0)$ then for a given $\epsilon_y(x,y)$, $\epsilon_x(x,y)$ needs to be large enough such that for some $\mu\in(0,1)$, $\inertia(\grad_{zz}f(x,y)+\mu E(x,y))\ne(n_x,n_y,0)$. 
	\end{itemize}
\end{enumerate}
We use these results in Section \ref{sc:alg-development-numerical-examples} to present an efficient way to numerically construct these functions.

\begin{proof}[Proof of Theorem \ref{th:stab-newton-minmax}]
The fact that the \eqref{eq:lqac} can always be satisfied is straightforward: as $\grad_{zz}f(x,y)$ is differentiable, its eigenvalues are bounded and can be made to have the desired inertia by taking sufficiently large (but finite) values of $\epsilon_x(x,y)$ and $\epsilon_y(x,y)$. Moreover, from our assumption that $\grad_{zz}f(x,y)$ and $\grad_{yy}f(x,y)$ are invertible, $(x,y)$ is a local minmax point if and only if $(x,y)$ satisfy the second order sufficient in \eqref{eq:secondorderuncont}; this is implied by the second order necessary conditions for local minmax in \cite{jin_what_2019}.
	
Using the same reasoning as in Theorem \ref{th:stab-newton-minimization}, as the \eqref{eq:lqac} hold then $(\grad_{zz}f(x,y)+E(x,y))$ is nonsingular and the Jacobian of the dynamical system \eqref{eq:newton} at $(x,y)$ is
\begin{equation}\label{eq:jacobnewton}
I-(\grad_{zz}f(x,y)+E(x,y))\inv \grad_{zz}f(x,y).
\end{equation}
Therefore, we can also use the same reasoning as in the proof of Theorem \ref{th:stab-newton-minimization} to conclude that $(x,y)$ is a LAS equilibrium point of \eqref{eq:newton} if $\grad_{zz}f(x,y) + \mu E(x,y)$ is nonsingular $\forall \mu \in [0, 2]$. Conversely, $(x,y)$ is an unstable equilibrium point of \eqref{eq:newton} if $\grad_{zz}f(x,y) + \mu E(x,y)$ is singular  for some $\mu \in (0, 2)$.

For the rest of the proof, it will be useful to have defined the function
\begin{equation}\label{eq:R_function_for_unconstrained}
R(\mu)=\grad_{xx}f(x,y) - \grad_{xy}f(x,y) (\grad_{yy}f(x,y) - \mu\epsilon_y(x,y) I)^{-1} \grad_{yx}f(x,y) + \mu\epsilon_xI
\end{equation}
and to drop the inputs $(x,y)$ from the expressions in order to shorten them.

\bigskip

Let us start by proving the statement for the case when $(x,y)$ is a local minmax, in which case the \eqref{eq:lqac} hold with $\epsilon_y=\epsilon_x=0$. We will prove that if
\begin{equation}\label{eq:minmax-eps-cond}
\epsilon_x\ge \lambda_{min}(\epsilon_y \grad_{xy}f \grad_{yy}f^{-2} \grad_{yx}f).
\end{equation}
then $(x,y)$ is a LAS equilibrium point of \eqref{eq:newton}.
To prove it, we will show \eqref{eq:minmax-eps-cond} ensures that $\grad_{zz}f + \mu E$ is nonsingular $\forall\ \mu\ge0$. First, as $ \grad_{yy}f \prec 0 $, $ \mu \geq 0 $, and $ \epsilon_y \geq 0 $, we have $ \grad_{yy}f - \mu \epsilon_y I \prec 0 $ and is thus nonsingular. Second, let us show that the condition \eqref{eq:minmax-eps-cond}
implies that for any vector $v$
\begin{align} \label{eq:eigenvalue_mu_smallest}
&\min_{\mu\in[0,2]}v'R(\mu)v=v'R(0)v.
\end{align}
Taking the derivative of $v'R(\mu)v$ with respect to $\mu$ we obtain
\begin{equation*} 
v'\Big(\epsilon_x I - \epsilon_y \grad_{xy}f (\grad_{yy}f - \mu \epsilon_y I)^{-2} \grad_{yx}f \Big)v\succ v'\Big(\epsilon_x I - \epsilon_y \grad_{xy}f \grad_{yy}f\inv[2] \grad_{yx}f \Big)v
\end{equation*}
in which we use the the fact that $\grad_{yy}f^{-2}\succeq(\grad_{yy}f - \mu \epsilon_y I)^{-2}$ for all $\mu\ge0$ as $ \grad_{yy}f \prec 0 $, and $ \epsilon_y \geq 0 $. Therefore, if \eqref{eq:minmax-eps-cond} holds, the derivative of $v'R(\mu)v$ with respect to $\mu$ is non-negative,  thus the cost does not decrease with $\mu$, which implies that the minimum is obtained for $\mu=0$, which proves \eqref{eq:eigenvalue_mu_smallest}. Therefore if $ \epsilon_x$ and $\epsilon_y$ are chosen to satisfy \eqref{eq:minmax-eps-cond}, then $\forall \mu \in [0,2]$ it holds that $R(\mu)\succeq R(0)\succ 0I$, where the second inequality comes from the second order sufficient conditions for unconstrained minmax \eqref{eq:secondorderuncont}. As neither $ \grad_{yy}f - \mu \epsilon_y I \prec 0 $ nor $R(\mu)$ are singular for $\mu\in[0,2]$, Haynsworth inertia additivity formula \cite[Theorem 1.6]{zhang_schur_2005}  implies that $\grad_{zz}f + \mu E$ is nonsingular $\forall \mu\in[0,2]$, and therefore $(x,y)$ is a LAS equilibrium point of \eqref{eq:newton}.

\bigskip

Now the second part, let us prove the statement for the case in which $(x,y)$ is not a local minmax. We will show that for every $\epsilon_y$ such that $\in(\grad_{yy}f-\epsilon_yI)=(0,n_y,0)$ for any large enough $\epsilon_x$, the \eqref{eq:lqac} are satisfied and
\begin{equation}
\grad_{zz}f+\mu \diag(\epsilon_x \oneb_{n_x}, -\epsilon_y \oneb_{n_y})=\grad_{zz}f+\mu E
\end{equation}
is singular for some $\mu\in(0,1)$, which in turn guarantees that $(x,y)$ is an unstable equilibrium point of \eqref{eq:newton} (see discussion in the beginning of the proof).

If $\inertia(\grad_{zz}f)\neq(n_x,n_y,0)$, then any large enough value of $\epsilon_x$ such that \eqref{eq:lqac} holds is enough to guarantee that $\grad_{zz}f+\mu E$ is singular for some $\mu\in(0,1)$. The proof is straightforward: If $\inertia(\grad_{zz}f)\neq(n_x,n_y,0)$ and $\inertia(\grad_{zz}f+E)=(n_x,n_y,0)$ (from the \eqref{eq:lqac}), then, by continuity of the eigenvalue $\exists \mu \in (0,1)$ such that $\grad_{zz}f+\mu E$ is singular.

If $\inertia(\grad_{zz}f)=(n_x,n_y,0)$ but $\inertia(\grad_{yy}f)\neq(0,n_y,0)$, then the \eqref{eq:lqac} is not enough to guarantee that $(x,y)$ is an unstable equilibrium point. However, it is possible to guarantee instability. The proof is the following.

Let $\mu^*$ be the largest $\mu \in (0,1)$ such that $\grad_{yy}f-\mu \epsilon_yI$ is singular. We know that this point exists because, on the one hand, by assumption $\grad_{yy}f$ is invertible (and therefore $\mu^*>0$), and on the other hand, we know that $\grad_{yy}f \nprec 0$ and that $\grad_{yy}f -\epsilon_yI \prec 0$ by construction (and therefore $\mu^*<1$). 

Now take any $\bar \mu\in(0,\mu^*)$ such that $\grad_{yy}f -\bar\mu\epsilon_yI$ is invertible (there are uncountable many). Suppose there exists $\bar \epsilon$ such that for any $\epsilon_x\ge\bar \epsilon$, the \eqref{eq:lqac} hold and $\inertia(\grad_{zz}f+\bar\mu E)\ne(n_x,n_y,0)$. If such $\bar \epsilon$ exists, then, by the continuity of the eigenvalues, if $\inertia(\grad_{zz}f+\bar\mu E)\ne(n_x,n_y,0)$ this means that $\grad_{zz}f+\mu E$ is singular for some $\mu\in(0,\bar \mu]$. 

So, to conclude the proof, we just need to show the existence of such $\bar \epsilon$. Take any $\epsilon_x$ such that $\inertia(\grad_{zz}f+\bar\mu E)=(n_x,n_y,0)$ (otherwise the proof is tautological). From Haynsworth inertia additivity formula, we have that 
$$
\inertia(\grad_{zz}f+\bar\mu E)= \inertia(R(\bar\mu)) + \inertia(\grad_{yy}f -\bar\mu\epsilon_yI)
$$
with $\inertia(R(\bar\mu))=(n_x-k,k,0)$ and $\inertia(\grad_{yy}f -\bar\mu\epsilon_yI)=(k,n_y-k,0)$ for some $k\in\{1,\dots,\min(n_x,n_y)\}$. On the one hand, it is straightforward to establish that $\exists  \bar\epsilon_1$ such that if $\epsilon_x\ge\bar\epsilon_1$, then $\inertia(R(\bar\mu))\neq(n_x-k,k,0)$, which means that $\inertia(\grad_{zz}f+\bar\mu E)\ne(n_x,n_y,0)$. On the other hand, $\exists  \bar\epsilon_2$ such that if $\epsilon_x\ge\bar\epsilon_2$, then $\inertia(\grad_{zz}f+\mu E)=(n_x,n_y,0)$. Therefore, we can define $\bar\epsilon=\max(\bar\epsilon_1,\bar\epsilon_2)$, which concludes the proof  
\end{proof}

\subsection{Constrained minmax} \label{sc:minmax-constrained}

We now consider the case with more general constraint sets involving equality and inequality constraints of the form
\begin{equation}\label{eq:constraints}
\begin{split}
&\Xcal=\{x\in \eR^{n_x}: G_x(x)=\zerob,F_x(x)\le \zerob\} \quad \tand\\
&\Ycal(x)=\{y\in\eR^{n_y}: G_y(x,y)=\zerob,F_y(x,y)\le \zerob\}
\end{split}
\end{equation}
where the functions $G_x : \eR^{n_x} \to \eR^{l_x}$, $F_x : \eR^{n_x} \to \eR^{m_x}$, $G_y: \eR^{n_x} \times \eR^{n_y} \to \eR^{l_y}$ and $F_y: \eR^{n_x} \times \eR^{n_y} \to \eR^{m_y}$ are all twice continuously differentiable. Similar to what we did in Section \ref{sc:minimization-constrained}, it will be convenient for the development of the primal-dual interior-point method to use slack variables and rewrite the constrained minmax \eqref{eq:minmaxconstrained} as
\begin{equation}\label{eq:minmaxconstrained_slack}
\min_{x,s_x: G_x(x)=\zerob,F_x(x)+s_x=\zerob,s_x\ge \zerob}\quad\max_{y,s_y: G_y(x,y)=\zerob,F_y(x,y)+s_y=\zerob,s_y\ge \zerob} f(x,y).
\end{equation}
where $s_x\in\eR^{m_x}$ and $s_y\in\eR^{m_y}$.

Similar to what we have done in the unconstrained case, we want to present second order conditions to determine if a point is a constrained local minmax. In order to do so, we need to extend some fundamental concepts of constrained minimization to constrained minmax optimization. The function 
\begin{equation*}
L(z):=f(x,y) + \nu_x'G_x(x) + \lambda_x'(F_x(x)+s_x)  +  \nu_y'G_y(x,y) - \lambda_y'(F_y(x,y)+s_y),
\end{equation*}
will play an equivalent role as the Lagrangian with $(\nu_x,\nu_y,\lambda_x,\lambda_y)$ as the equivalent of Lagrange multipliers; we use the shorthand notation $z=(x,s_x,y,s_y,\nu_y,\lambda_y,\nu_x,\lambda_x)$. Furthermore, we use the following definition of linear independence constraint qualifications (LICQ) and of strict complementarity for minmax optimization:
\begin{definition}[LICQ and strict complementarity for minmax]
	Let the sets of active inequality constraints for the minimization and maximization be defined, respectively, by
	\begin{equation*}
	\begin{split}
	&\Acal_x(x)=\{i=1,\dots,m_x : F_x^{(i)}(x)=0\} \tand \\
	&\Acal_y(x,y)=\{i=1,\dots,m_y : F_y^{(i)}(x,y)=0\}
	\end{split}
	\end{equation*}
	where $F_x^{(i)}(x)$ and $F_y^{(i)}(x,y)$ denote the i\textsuperscript{th} element of $F_x(x)$ and $F_y(x,y)$. Then:
	\begin{itemize}
		\item  The linear independence constraint qualification (LICQ) is said to hold at $z$ if the vectors in the sets
		\begin{gather*}
		\{\grad_x G_x^{(i)}(x), i=1,\dots,l_x\}\bigcup\{\grad_x F_x^{(i)}(x), i \in \Acal_x(x) \} \tand\\
		\{\grad_y G_y^{(i)}(x,y), i=1,\dots,l_y\}\bigcup\{\grad_y F_y^{(i)}(x,y), i \in \Acal_y(x,y) \}
		\end{gather*}
		are linearly independent.
		\item Strict complementarity is said to hold at $z$ if $\lambda_y^{(i)}>0 \ \forall i\in \Acal_y(x,y)$ and $\lambda_x^{(i)}>0 \ \forall i\in \Acal_x(x)$ 
	\end{itemize}
\end{definition}

We have almost all the ingredients to present the second order conditions for constrained minimization. For the unconstrained minmax optimization, the second order condition in Proposition \ref{prop:second-order-unconstrained-minmax} required that gradients ($\grad_{x}f(x,y)$ and $\grad_{y}f(x,y)$) were equal to zero and that Hessians ($\grad_{zz}f(x,y)$ and $\grad_{yy}f(x,y)$) had a particular inertia. Analogously to what was the case for the constrained minimization in Section \ref{sc:minimization-constrained}, if it were not for the inequality constraints in \eqref{eq:constraints}, we would be able to state the second order conditions using gradients and Hessians of $L(z)$. The inequality constraints make the statement a bit more complicated. The role of the gradient will be played by
\begin{equation*}
g(z,b):=\begin{bmatrix}
&\grad_{x} L(z) \\
& \lambda_x \odot s_x - b \oneb \\
&\grad_{y} L(z) \\
& -\lambda_y \odot s_y + b \oneb  \\
& G_y(x,y)  \\
& -F_y(x,y)- s_y\\
& G_x(x)  \\
& F_x(x) +s_x
\end{bmatrix}
\end{equation*}
where $\odot$ denotes the element wise Hadamard product of two vectors and $b\ge0$  the barrier parameter, which is the extension to minmax of the function $g(\cdot)$ defined in \eqref{eq:kkt_min} for the minimization. The role of
$\grad_{yy}f(x,y)$ will be played by
\begin{subequations}\label{eq:definingblocs}
\begin{equation}
H_{yy}f(z)=
\begin{bmatrix}
\grad_{yy} L(z) & \zerob &  \grad_y G_y(x,y) & -\grad_y F_y(x,y) \\
\zerob & -\diag(\lambda_y) &\zerob & -\diag(s_y^{1/2})  \\
\grad_y G_y(x,y)' &\zerob &\zerob &\zerob  \\
-\grad_y F_y(x,y)' & -\diag(s_y^{1/2}) & \zerob & \zerob
\end{bmatrix},
\end{equation}
while the role of $\grad_{zz}f(x,y)$ will be played by
\begin{equation} \label{eq:definingblocs_b}
H_{zz}f(z)=
\begin{bmatrix}
H_{xx}f(z) & H_{xy}f(z) &  H_{x\lambda}f(z)\\
H_{xy}f(z)' &  H_{yy}f(z) & \zerob \\
H_{x\lambda}f(z)' & \zerob & \zerob
\end{bmatrix}
\end{equation}
with blocks defined by
\begin{equation} 
\begin{split}
&H_{xy}f(z)=
\begin{bmatrix}
\grad_{xy} L(z) & \zerob & \grad_x G_y(x,y) &-\grad_x F_y(x,y)\\
\zerob & \zerob & \zerob &\zerob \\
\end{bmatrix}
\\
&H_{xx}f(z)=
\begin{bmatrix}
\grad_{xx} L(z) & \zerob \\
\zerob & \diag(\lambda_x) 
\end{bmatrix} 
\quad
H_{x\lambda}f(z)=
\begin{bmatrix}
\grad_x G_x(x) & \grad_x F_x(x)\\
\zerob & \diag(s_x^{1/2})
\end{bmatrix}
\end{split}
\end{equation}
\end{subequations}
\begin{proposition}[Second order sufficient conditions for constrained minmax] \label{prop:second-order-constrained-minmax}
Let $z$ be an equilibrium point in the sense that $g(z,0)=\zerob$ with $\lambda_y,\lambda_x,s_y,s_x\ge\zerob$. If the LICQ and strict complementarity hold at $z$ and
\begin{equation} \label{eq:second_order_inertia_constrained}
\begin{split}
\inertia( H_{yy}f(z))=(l_y+&m_y,n_y+m_y,0) \tand\\
\inertia( H_{zz}f(z))=(n_x+m_x+l_y+&m_y,l_x+m_x+n_y+m_y,0)
\end{split}
\end{equation}
then $(x,y)$ is a local minmax of \eqref{eq:minmaxconstrained}. 
\end{proposition}

Similar to what was the case for the second order sufficient conditions for unconstrained minmax in Proposition \ref{prop:second-order-unconstrained-minmax}, the conditions in \eqref{eq:second_order_inertia_constrained} are not symmetric, highlighting that there is a distinction between the minimizer and maximizer. Moreover, similar to the second order sufficient conditions for unconstrained minmax in Proposition \ref{prop:second-order-unconstrained-minmax}, one can interpret the second order sufficient conditions for constrained minmax as saying that the optimization $\max_{y\in\Ycal(x)} f(x,y)$ is strongly concave in a neighborhood around $(x,y)$ and that the optimization $\min_{x\in \Xcal} \phi(x)$ with $\phi(x):= \max_{\tilde y\in\Ycal(x):\norm{y-\tilde y}<\delta} f(x,\tilde y)$ is strongly convex in a neighborhood around $(x,y)$ for some $\delta>0$.

The conditions for Proposition \ref{prop:second-order-constrained-minmax} are slightly stricter than the ones in \cite{dai_optimality_2020} as we require strict complementarity and LICQ both for the max and the min. However, our conditions allow us to verify whether a point is a local minmax using the inertia, instead of having to compute solution cones. We prove that given these stricter assumptions our conditions are equivalent to those in \cite{dai_optimality_2020} in Appendix \ref{app:proof_propo_secon}.

\subsubsection{Primal-dual interior-point method}

Let $d_z=(d_x,d_{s_x},d_y,d_{s_y},d_{\nu_y},d_{\lambda_y},d_{\nu_x},d_{\lambda_x})$ be a shorthand notation to designate the update direction of the variables $z=(x,s_x,y,s_y,\nu_y,\lambda_y,\nu_x,\lambda_x)$. Similar to the basic primal-dual interior-point method introduced in Section \ref{sc:minimization-constrained}, a basic primal-dual interior-point method for minmax finds a candidate solution to  \eqref{eq:minmaxconstrained_slack} using the iterations 
\begin{equation}\label{eq:basic_ip}
z^+=z+\alpha d_z=z-\alpha \grad_{z}g(z,b)'\,\inv g(z,b)
\end{equation}
where the barrier parameter $b$ is slowly decreased to $0$, so that $z$ converges to a root of $g(z,0)=\zerob$ while $\alpha\in (0,1]$ is chosen at each step such that the feasibility conditions $\lambda_y,\lambda_x,s_y,s_x>0$ hold. We want to modify this basic primal-dual interior-point so it satisfies the properties P1, P2 and P3.

In order to obtain property P1, we propose to obtain $d_z$ from the solution of a quadratic program that locally approximates \eqref{eq:minmaxconstrained_slack}. Using equivalent arguments as in the development of the quadratic program \eqref{eq:min_QP} for the constrained minimization in Section \ref{sc:minimization-constrained}, we obtain that the objective function should be
\begin{multline*}
K(d_x,d_{s_x},d_y,d_{s_y})=L(z)+\grad_{x}L(z)'d_x+(\lambda_x-b\oneb\oslash s_x)' d_{s_x} + \grad_{y}L(z)'d_y
\\-(\lambda_y-b\oneb\oslash s_y)' d_{s_y} + d_x'\grad_{xy}L(z)d_y 
+\frac{1}{2}d_x'(\grad_{xx}L(z)+\epsilon_x(z)I)d_x
\\+\frac{1}{2}d_{s_x}'\diag(\lambda_x\oslash s_x)d_{s_x}+\frac{1}{2}d_y'(\grad_{yy}L(z)-\epsilon_y(z)I)d_y-\frac{1}{2}d_{s_y}'\diag(\lambda_y\oslash s_y)d_{s_y},
\end{multline*}
where $\epsilon_x(z)\ge0$ and $\epsilon_y(z)\ge0$ are scalar and $\oslash$ designates the element wise division of two vectors. The feasible sets $d\Xcal$ for $(d_x,d_{s_x})$ and the set-valued map that defines a feasible set $d\Ycal(d_x)$ for $(d_y,d_{s_y})$ are obtained from the first order linearization of the functions in $\Xcal$ and $\Ycal(d_y)$ and are given by
\begin{align*}
&d\Xcal=\{(d_x,d_{s_x})\in\eR^{n_x}\times \eR^{m_x}:G_x(x)+\grad_xG_x(x)'d_x=\zerob,\\
& \hspace{170pt}F_x(x)+s_x+\grad_xF_x(x)'d_x+d_{s_x}=\zerob\}
\\
&d\Ycal(d_x)=\{(d_y,d_{s_y})\in\eR^{n_y}\times \eR^{m_y}: G_y(x,y)+\grad_xG_y(x,y)'d_x+\grad_yG_y(x,y)'d_y
\\
& \hspace{55pt}=\zerob ,F_y(x,y)+s_y+\grad_xF_y(x,y)'d_x+\grad_yF_y(x,y)'d_y+d_{s_y}=\zerob\}.
\end{align*}
If $\grad_xG_x(x)$ and $\grad_yG_y(x,y)$ have linearly independent columns, we propose to obtain  $(d_x,d_{s_x},d_y,d_{s_y})$ as the optimizers and $(d_{\nu_y},d_{\lambda_y},d_{\nu_x},d_{\lambda_x})$ the associated Lagrange multipliers of the minmax optimization
\begin{equation} \label{eq:QP}
\min_{\bar d_x,\bar d_{s_x}\in d\Xcal}\quad\max_{\bar d_y,\bar d_{s_y}\in d\Ycal(\bar d_x)} K(\bar d_x,\bar d_{s_x},\bar d_y,\bar d_{s_y})
\end{equation}
where $\epsilon_x(z)$ and $\epsilon_y(z)$ are chosen such that the solution to \eqref{eq:QP} is unique. We can apply to \eqref{eq:QP} the second order condition from Proposition \ref{prop:second-order-constrained-minmax} and obtain that $\epsilon_x(z)$ and $\epsilon_y(z)$ need to be chosen to satisfy
\begin{equation}\label{eq:conslqac}
\tag{ConsLQAC}
\begin{split}
\inertia(J_{yy}f(z)-E_y(z))=(l_y+&m_y,n_y+m_y,0) \tand\\
\inertia(J_{zz}f(z)+E(z))=(n_x+m_x+l_y+&m_y,l_x+m_x+n_y+m_y,0)
\end{split}
\end{equation}
where $E_y(z):= \diag(\epsilon_y(z) \oneb_{n_y}, \zerob_{l_y+2m_y})$ and $E(z):=\diag(\epsilon_x(z)\oneb_{n_x},\zerob_{m_x},-\epsilon_y(z) \oneb_{n_y}, \zerob_{l_y+2m_y+l_x+m_x})$; $J_{zz}f(z)$ is the equivalent of the matrix defined in \eqref{eq:definingblocs_b} for the problem \eqref{eq:QP} and can be shown to be equal to
\begin{equation}\label{eq:equivalenceJ}
J_{zz}f(z)=S^{-1/2}H_{zz}f(z)S^{-1/2}=S\inv \grad_{z}g(z,b)'.
\end{equation}
with $S=\diag(\oneb_{n_x},s_x,\oneb_{n_y},s_y,\oneb_{l_y+m_y+l_x+m_x})$; $J_{yy}f(z)$ is the equivalent partition of $J_{zz}f(z)$ as $H_{yy}(z)$ is of $H_{zz}(z)$. We will call these conditions the Constrained Local Quadratic Approximation Conditions \eqref{eq:conslqac}. In this case, it is straightforward to show that modifying the basic primal-dual interior-point iterations in \eqref{eq:basic_ip} by taking $d_z$ from the solution of \eqref{eq:QP} leads to the iterations
\begin{equation}\label{eq:interiorpoint}
z^+=z+\alpha d_z=z-\alpha(J_{zz}f(z)+E(z))\inv S\inv g(z,b).
\end{equation}

Analogously to what was the case in unconstrained minmax optimization, choosing $\epsilon_x(z)$ and $\epsilon_y(z)$ such that the \eqref{eq:conslqac} hold is not sufficient to guarantee that P2 and P3 hold for the modified primal-dual interior-point method (a counter example can be found in Section \ref{sc:alg-development-numerical-examples-MPC}). Our next theorem is the extensions of Theorem \ref{th:stab-newton-minmax} to the modified primal-dual interior-point and has the equivalent consequences:  For property P3 to hold, as long as $\epsilon_x(z)$ is large enough, taking $\epsilon_y(z)>0$ will not impair the algorithm's capacity to converge towards a local minmax; this can be useful, for instance, if $\inertia(J_{zz})$ has an eigenvalue close to $0$. For property P2 to hold, in order to guarantee that the modified primal-dual interior-point method cannot converge towards an equilibrium point that is not local minmax, the \eqref{eq:conslqac} are sufficient only whenever $\inertia(J_{zz}f(z))\ne(n_x+m_x+l_y+m_y,l_x+m_x+n_y+m_y,0)$. Otherwise, $\epsilon_x(z)$ needs to be taken large enough such that $\inertia(J_{zz}f(z)+\mu E(z))\ne(n_x+m_x+l_y+m_y,l_x+m_x+n_y+m_y,0)$ for some $\mu \in (0,1)$.

\begin{theorem}[Stability and instability of modified primal-dual interior-point method for constrained minmax] \label{th:stab-ip-minmax}
	Let $\alpha=1$ and $(z,b)$ with $b>0$, be an equilibrium point in the sense that $g(z,b)=\zerob$. Assume the LICQ hold at $z$, that $J_{zz}f(z)$ and $J_{yy}f(z)$ are invertible, and that $J_{zz}f(\cdot)$ is differentiable in a neighborhood around $z$. Then there exists functions $\epsilon_x(\cdot)$ and $\epsilon_y(\cdot)$ that are constant in a neighborhood around $z$, satisfy the \eqref{eq:conslqac} at $z$  and guarantee that if:		
	\begin{enumerate}
		\item[i)] $z$ is a local minmax of \eqref{eq:minmaxconstrained_slack}, then it is a LAS equilibrium of \eqref{eq:interiorpoint}.
		\item[ii)] $z$ is not a local minmax of \eqref{eq:minmaxconstrained_slack}, then it is an unstable equilibrium of \eqref{eq:interiorpoint}.
	\end{enumerate}		
\end{theorem}

\begin{proof}
Let us define the partitions, $J_{xx}f(z)$, $J_{yx}f(z)$, and $J_{x\lambda}f(z)$ of $J_{zz}f(z)$ analogously to the partitions $ H_{xx}f(z)$, $ H_{yx}f(z)$, and $ H_{x\lambda}f(z)$ of $ H_{zz}f(z)$.	
	
Using the same arguments as in the proof of Theorem \ref{th:stab-newton-minimization}, we conclude that the Jacobian of the dynamic system \eqref{eq:interiorpoint} around a point $z$ such that $g(z,b)=\zerob$ is
\begin{equation}\label{eq:jacob_ip}
I- \alpha \Big(J_{zz}f(z)+E(z)\Big)\inv S\inv \grad_zg(z,b)'=I- \alpha \Big(J_{zz}f(z)+E(z)\Big)\inv J_{zz}f(z)
\end{equation}

Moreover from \eqref{eq:equivalenceJ} we have that $\inertia(H_{zz}f(z))=\inertia(S^{1/2}J_{zz}f(z)S^{1/2})$. Using Sylvester's law of inertia \cite[Theorem 1.5]{zhang_schur_2005}, this simplifies to $\inertia(H_{zz}f(z))=\inertia(J_{zz}f(z))$. If a point $z$ is such that $g(z,b)=\zerob$, then one can check \eqref{eq:second_order_inertia_constrained} using $J_{zz}f(z)$ and $J_{yy}f(z)$.

Let us define the matrices
\begin{subequations}
	\begin{equation}
	R_y(\mu)=Z_y(z)'\begin{bmatrix}
	\grad_{yy} L(z) -\epsilon(z)\mu I & \zerob \\
	\zerob & -\diag(\lambda_y\oslash s_y)
	\end{bmatrix}Z_y(z)
	\end{equation}
\begin{equation} \label{eq:S_matrix_constrained}
	R_x(\mu)=Z_x(z)'\Big(J_{xx}f(z) - J_{xy}f(z) (J_{yy}f(z) - \mu E_y(z))^{-1} J_{yx}f(x,y) + \mu E_x(z)\Big)Z_x(z)
\end{equation}
\end{subequations}
where $Z_{y}(z)\in\eR^{n_y+m_y,n_y-l_y}$ and $Z_{x}(z)\in\eR^{n_x+m_x,n_x-l_x}$ are any full column rank matrices such that
\begin{equation}
\begin{bmatrix}
\grad_{y} G_y(x,y) & -\grad_{y}F_y(x,y) \\
-I & 0
\end{bmatrix}\,Z_y(z)=\zerob \quad \tand \quad J_{x\lambda}f(z)'\,Z_x(z)=\zerob.
\end{equation}
Using the same reasoning as in the proof of Proposition \ref{prop:second-order-constrained-minmax} one can conclude that
\begin{gather*}
\inertia(J_{yy}f(z)-\mu E_y(z))=\inertia(R_y(\mu))+(l_y+m_y,l_y+m_y,0)\\
\inertia(J_{zz}f(z)+\mu E(z))=\inertia(R_x(\mu))+\inertia(J_{yy}f(z)-\mu E_y(z))+(l_x+m_x,l_x+m_x,0),
\end{gather*}
which implies that the \eqref{eq:conslqac} can be stated as
\begin{equation*}
R_y(1)\prec \zerob \quad \tand \quad R_x(1)\succ \zerob.
\end{equation*}
This means that the exact same arguments used in the proof of the unconstrained minmax in Theorem \ref{th:stab-newton-minmax} can be used for the constrained case. More specifically, each arguments with 
\begin{gather*}
\grad_{yy}f(x,y)-\epsilon_y(x,y)\mu I \\
\tand \\
\grad_{xx}f(x,y) - \grad_{xy}f(x,y) (\grad_{yy}f(x,y) - \mu\epsilon_y(x,y) I)^{-1} \grad_{yx}f(x,y) + \mu\epsilon_x(x,y)I.
\end{gather*}
has an analogous statement with $R_y(\mu)$ and $R_x(\mu)$, respectively. For the sake of completeness, we highlight the main points of the analogy.

\bigskip

First, when $z$ is such that \eqref{eq:second_order_inertia_constrained} holds, the sufficient condition for $z$ to be a LAS equilibrium point of \eqref{eq:interiorpoint} is that
\begin{equation} \label{eq:minmax-eps-cond_cons}
\grad_\mu R_x(0)=Z_x(z)'\Big(E_x(z)-J_{xy}f(z) J_{yy}f(z)^{-1}E_y(z)J_{yy}f(z)^{-1} J_{yx}f(z)\Big) Z_x(z)\succeq 0.
\end{equation}
The only extra argument needed is to show that condition \eqref{eq:minmax-eps-cond_cons} is always feasible for some $\epsilon_x(z)$ large enough. This is not evident as the matrix
$$
M:=-J_{xy}f(z) J_{yy}f(z)^{-1}E_y(z)J_{yy}f(z)^{-1} J_{yx}f(z)
$$
has size $(n_x+m_x)\times (n_x+m_x)$ while $E_x(z)$ only has $n_x$ nonzero elements in the diagonal. However, because of the structural zeros in $J_{xy}f(z)$ and $E_y(z)$, one can verify with some algebraic manipulation that $\rank(M):=r\le\min(n_x,n_y)$. Let $\Lambda$ be the matrix with eigenvalues of $M$ in decreasing order and $V$ its associated eigenvectors such that $M=V\Lambda V'$. We can partition $V$ into $V_1$ of size $(r,r)$ associated to the nonzero eigenvalues of $M$ and $V_2=I_{n_x+m_x-r}$. This partition means that $E_x(z)=V'E_x(z)V$, which means on can conclude that
\begin{equation*}
\grad_\mu R_x(\mu)=Z_x(z)'V'\Big(E_x(z)+\Lambda\Big)VZ_x(z),
\end{equation*}
which implies that one can always take $\epsilon_x$ large enough such that for each negative diagonal entries of $\Lambda$, the equivalent diagonal element of $(E_x(z)+\Lambda)$ is positive.

\bigskip

Now the second part, let us prove the statement when $z$ is such that the second order conditions in \eqref{eq:second_order_inertia_constrained} do not hold. We need to prove that 
\begin{equation} \label{eq:gradyymubar}
J_{zz}f(z)+\mu E(z)
\end{equation}
is singular for some $\mu\in(0,1)$. On the one hand, using the same analysis as in the proof of Theorem \ref{th:stab-newton-minmax}, we conclude that the \eqref{eq:conslqac} are sufficient to guarantee that $z$ is an unstable equilibrium point of \eqref{eq:interiorpoint} if $\inertia(J_{zz}f(z))\neq(n_x+m_x+l_y+m_y,l_x+m_x+n_y+m_y,0)$. On the other hand, if $\inertia(J_{zz}f(z))=(n_x+m_x+l_y+m_y,l_x+m_x+n_y+m_y,0)$, than we can guarantee that by taking $\epsilon_x$ sufficiently large, there is a $\mu\in(0,1)$ such that $\inertia(J_{zz}f(z)+\mu E)\neq(n_x+m_x+l_y+m_y,l_x+m_x+n_y+m_y,0)$, which means that $z$ is an unstable equilibrium point of \eqref{eq:interiorpoint}. This concludes the proof. 
\end{proof}

\section{Algorithmic development and numerical examples} \label{sc:alg-development-numerical-examples}

The following algorithm combines the result of the previous section to propose a method for selecting $\epsilon_x(z)$ and $\epsilon_y(z)$ that satisfies the \eqref{eq:conslqac} and guarantees the stability properties of Theorem \ref{th:stab-ip-minmax}. We only state the algorithm for the constrained case, its specialization to the unconstrained case is straightforward. In order to keep the algorithm more simple and to highlight the instability property, we chose to use the functions $\epsilon_y(\cdot)=\epsilon_x(\cdot)=0$ whenever the algorithm is near a local minmax. 

\begin{algorithm}[H]
	\begin{algorithmic}[1]
		\Require  An initial point $z=(x,s_x,y,s_y,\nu_y,\lambda_y,\nu_x,\lambda_x)$, an initial barrier parameter value $b$, a barrier reduction factor $\sigma\in(0,1)$, a stopping accuracy $\delta_s\ge0$, a $\delta_\epsilon>0$ that defines a neighborhood for stopping to adjust $\epsilon_x$ and $\epsilon_y$.
		\While{$\norm{g(z,b)}_\infty>\delta_s$}
		\If{$\norm{g(z,b)}_\infty>\delta_\epsilon$}
		\State $\epsilon_x\leftarrow0, \epsilon_y \leftarrow 0$
		\If{\eqref{eq:conslqac} cannot be satisfied with $\epsilon_y=\epsilon_x=0$}
			\State Increase $\epsilon_y$ until
			$$
			\inertia(J_{yy}f(z)-E_y)=(l_y+m_y,n_y+m_y,0)
			$$
			\State Increase $\epsilon_x$ until
			$$
			\inertia(J_{zz}f(z)+E)=(n_x+m_x+l_y+m_y,l_x+m_x+n_y+m_y,0)
			$$
			\If{$\inertia(J_{zz}f)=(n_x+m_x+l_y+m_y,l_x+m_x+n_y+m_y,0)$} \label{algline:instability}
				\State Increase $\epsilon_x$ until, for some value of $\mu\in(0,1)$,
			$$
			\inertia(J_{zz}f(z)+\mu E(z))\neq(n_x+m_x+l_y+m_y,l_x+m_x+n_y+m_y,0)
			$$							
			\EndIf
		\EndIf
		\EndIf
		\State Compute a new $z$ using the equation
		$$
		z\leftarrow z- \alpha \Big(J_{zz}f(z)+E\Big)\inv S\inv g(z,b)
		$$
		\Statex \hspace{\algorithmicindent} \hspace{\algorithmicindent}where $\alpha\in(0,1]$ is selected such that the feasibility conditions 
		\Statex \hspace{\algorithmicindent} \hspace{\algorithmicindent} $\lambda_y ,\lambda_x,s_y,s_x> \zerob$ hold.
		\If{$\norm{g(z,b)}_\infty\le b$}
		\State $b\leftarrow\sigma\,b$		
		\EndIf		
		\EndWhile
	\end{algorithmic}
	\caption{Primal-dual interior-point method for minmax}
	\label{alg:ip-minmax}
\end{algorithm}

\begin{proposition}[Construction of the modified primal-dual interior-point method] \label{prop:construction_of_epsilons_numerically}
Algorithm \ref{alg:ip-minmax} generates functions $\epsilon_x(\cdot)$ and $\epsilon_y(\cdot)$ that satisfy the conditions of Theorem \ref{th:stab-ip-minmax} in the neighborhood of any equilibrium point $z^*$ that satisfy the assumptions of Theorem \ref{th:stab-ip-minmax}.  
\end{proposition}
\begin{proof}
For each $z$, Algorithm \ref{alg:ip-minmax} produces values of $\epsilon_x$ and $\epsilon_y$ that only depend on $z$, therefore it implicitly constructs functions $\epsilon_x(\cdot)$ and $\epsilon_y(\cdot)$. Moreover, $\epsilon_x(\cdot)$ and $\epsilon_y(\cdot)$ are such that either the stability condition \eqref{eq:minmax-eps-cond_cons} or the instability condition  \eqref{eq:gradyymubar} are satisfied for each $z$, therefore they are satisfied in the neighborhood of any equilibrium point $z^*$. Finally, $\epsilon_x(\cdot)$ and $\epsilon_y(\cdot)$ are constant in a neighborhood around each equilibrium point as the values of $\epsilon_x$ and $\epsilon_y$ are not adjusted when $\norm{g(z,b)}_{\infty}\le \delta_\epsilon$.
\end{proof}


In Algorithm \ref{alg:ip-minmax}, for each $z$, $(\epsilon_x,\epsilon_y)$ is chosen to satisfy the conditions of Theorem \ref{th:stab-ip-minmax}, and therefore generate the desired stability and instability. This means that the algorithm essentially guarantees that the modified primal-dual interior-point method can only converge to an equilibrium point if such point is a local minmax. A key point of the algorithm is that it only uses the inertia of matrices, which can be efficiently computed using either the LBLt or LDLt decomposition, as we further detail in the following remark.

\begin{remark}[Computing the inertia] It is not necessary to actually compute the eigenvalues of $J_{zz}f(z)$ in order to determine the inertia. A first option is to use the lower-triangular-block-lower-triangular-transpose (LBLt) decomposition \cite[Appendix A]{nocedal_numerical_2006}, which decomposes $J_{zz}f(z)$ into the product $LBL'$ where $L$ is a lower triangular matrix and $B$ a block diagonal one, the inertia of $B$ is the same as the inertia of $J_{zz}f(z)$.
	
Let $\Gamma=\diag(\gamma\oneb_{n_x+m_x},-\gamma\oneb_{n_y+m_y},\gamma\oneb_{l_y+m_y},-\gamma\oneb_{l_x+m_x})$, with $\gamma$ a small positive number. A second approach is to use the lower-triangular-diagonal-lower-triangular-transpose (LDLt) decomposition, to decompose $J_{zz}f(z)+\Gamma$ into the product $LDL'$ where $L$ is a lower triangular matrix and $D$ is a diagonal matrix; the inertia of $D$, which is given by the number of positive, negative and zero elements of the diagonal of $D$, gives the inertia of $J_{zz}f(z)+\Gamma$. The matrix $\Gamma$ introduces a distortion in the inertia but it helps to stabilize the computation of the LDLt decomposition, which tends to be faster than the LBLt decomposition. This is the approach we use in our implementation; it has been studied in primal-dual interior-point algorithms for minimization and the distortion introduced by $\Gamma$ tends to be compensated by a better numerical algorithm \cite{vanderbei_symmetric_1995,higham_modifying_1998}.\hfill $\square$
\end{remark}

\subsection{Benchmark example for unconstrained minmax} \label{sc:alg-development-numerical-examples-benchmark}
Consider the following functions
\begin{align*}
f_1(x,y)&=2x^2 -y^2 + 4xy + 4/3y^3 - 1/4y^4 \\
f_2(x,y)&=(4x^2-(y-3x+0.05x^3)^2-0.1y^4)\exp(-0.01(x^2+y^2)) \\
f_3(x,y)&=(x-0.5)(y-0.5)+\exp(-(x-0.25)^2-(y-0.75)^2) \\
f_4(x,y)&=xy.
\end{align*}
The first three have been used as examples in \cite{adolphs_local_2019,wang_solving_2019,mertikopoulos_optimistic_2019} respectively, whereas the fourth one is a well known case for a simple but challenging function to find the local minmax. These problems all satisfy the assumption of Theorem \ref{th:stab-newton-minmax} and have local minmax points. We have chosen these functions because, as we will show, they illustrate some interesting behaviors.

Our goal is to compare the performance of Algorithm \ref{alg:ip-minmax} to the performance of two well established algorithms. On the one hand, we look at the performance of a ``pure'' Newton algorithm, \ie using $\epsilon_x(\cdot)=\epsilon_y(\cdot)=0$. On the other hand, we look into the convergence of a Gradient Descent Ascent (GDA), \ie
\begin{equation*}
	\begin{split}
		x^+=x-\alpha_x \grad_{x}f(x,y) \\
		y^+=y+\alpha_y \grad_{y}f(x,y)
	\end{split}
\end{equation*}
where $\alpha_x$ and $\alpha_y$ are constant and different for each problem; we did our best to select the best values $\alpha_x$ and $\alpha_y$ for each problem.

Each algorithm is initialized 1000 times, using the same initialization for the three of them each time. We compare their convergence properties according to three criteria: the number of times the algorithm converged to an equilibrium point (eq.), the number of times it converged to a local minmax point (minmax) and the average number of iterations to converge to a local minmax point (iter). The algorithm is terminated when the infinity norm of the gradient is smaller than $\delta_s=10\inv[5]$ and we declare that they did not converge if it has not terminated in less than 500 iterations for the pure Newton and Algorithm \ref{alg:ip-minmax}, and 50 000 for GDA. The result of the comparison is displayed in Table \ref{tab:compa}. The key take away from these examples is that Algorithm \ref{alg:ip-minmax} never converges towards an equilibrium point that is not a local minmax, in contrast with the pure Newton method which is attracted to any equilibrium point. Here is a detailed observation from this comparison.

\begin{table}[t]
	\begin{tabular}{c|ccc|ccc|ccc}
		& \multicolumn{3}{c|}{Pure Newton} & \multicolumn{3}{c|}{GDA} & \multicolumn{3}{c}{Algorithm \ref{alg:ip-minmax}}  \\
		& eq. & minmax  & iter      &eq. & minmax  & iter & eq. & minmax  & iter     \\ \hline
		$f_1$	& 1000      & 1000   &   4.1    &  1000     &  1000     &  485    &  1000     &  1000     &  5.7    \\
		$f_2$	&  1000     &  665     &   7.3    & 976      &  976     &  18195     & 996     &  996     &  8.1    \\
		$f_3$	&  954     &  485     &  4.8     &  373    &  373     &  40936     &  709     &  709     &  7.1     \\
		$f_4$	&  1000     &  1000     &  1     &  0    &  0     &  --     &  1000     &  1000     &  1  
	\end{tabular}\
\caption{Comparing the performance of Pure Newton's method, Gradient Descent Ascent and Algorithm \ref{alg:ip-minmax}. We randombly generated 1000 initializations and used them to start the algorithms from different locations. ``eq'' refers to the number of times the algorithm converged to an equilibrium point, \ie a point that satisfy the first order condition (but not necessarily the second order ones), ``minmax`` refers to the number of times the algorithm converged to a local mimmax, ``iter'' refers to the number of iterations it took for the algorithm to converge to a minmax point.}
\label{tab:compa}
\end{table}
\begin{itemize}
	\item The pure Newton algorithm has good overall convergence for all the problems, but it also tends to often converge towards an equilibrium point that is not a local minmax problems. On the other hand, when the pure Newton converges to a local minmax, it does so in less iterations than the other two methods. This is expected when comparing to the GDA, as it is a first order method. Pure Newton algorithm converges in (slightly) less iterations than Algorithm \ref{alg:ip-minmax} because taking $\epsilon_x$ and $\epsilon_y$  different than $0$ hinders the superlinear convergence property of Newton's method.
	\item The GDA algorithm seems to enjoy the property of always converging towards a local minmax, and except for $f_3(\cdot)$ and $f_4(\cdot)$, it has good rate of convergence. However, GDA takes an exceptionally long number of iterations to converge. This is somehow expected from the fact that it is a first order method, and it is partially compensated by each iteration being more simple to compute. However, one must keep in mind that none of this takes into account the time that needs to be spent adjusting the step sizes until a good convergence rate can be obtained.
	\item At last, Algorithm \ref{alg:ip-minmax} is across the board the algorithm with better convergence towards local minmax, and it does so in the smallest number of iterations. As it was expected from the theory, Algorithm \ref{alg:ip-minmax} never converges towards an equilibrium point that is not a local minmax. From a numerical perspective, the biggest takeaway is that while our results are only about local convergence, the algorithm still enjoys good global convergence properties; only in $f_3(\cdot)$ it does not converge essentially $100\%$ of the time. 
	\item Function $f_4(\cdot)$ is  particularly interesting example. First, notice that the pure Newton converges in one iteration. This is expected as  the iterations are given by
	$$
	\begin{bmatrix}
		x^+\\y^+
	\end{bmatrix}=
		\begin{bmatrix}
		x\\y
	\end{bmatrix}-
	\begin{bmatrix}
	0 & 1 \\1 & 0
\end{bmatrix}\inv
	\begin{bmatrix}
	y\\x
\end{bmatrix}=
\begin{bmatrix}
	x\\y
\end{bmatrix}-
\begin{bmatrix}
	0 & 1 \\1 & 0
\end{bmatrix}\inv
\begin{bmatrix}
	0 & 1 \\1 & 0
\end{bmatrix}
\begin{bmatrix}
	x\\y
\end{bmatrix}=
\begin{bmatrix}
	0\\0
\end{bmatrix}.
	$$
This is in stark contrast with GDA which, as it is well known, diverges away from the local minmax. As for Algorithm \ref{alg:ip-minmax}, it converges even though it does not satisfy the assumptions of Theorem \ref{th:stab-newton-minmax}, further emphasizing that these are sufficient but not necessary conditions. Notice that Algorithm \ref{alg:ip-minmax} is not the same as the pure Newton as the Hessian will be modified with an $\epsilon_y(x,y)>0$ to guarantee that the portion of the Hessian associated to the maximization is negative definite.
\end{itemize}

\subsection{The homicidal chauffeur example for constrained minmax} \label{sc:alg-development-numerical-examples-MPC}

\begin{figure*}[t!]
	\centering
	\subfloat[While guaranteeing instability]{%
		\includegraphics[width=0.32\linewidth]{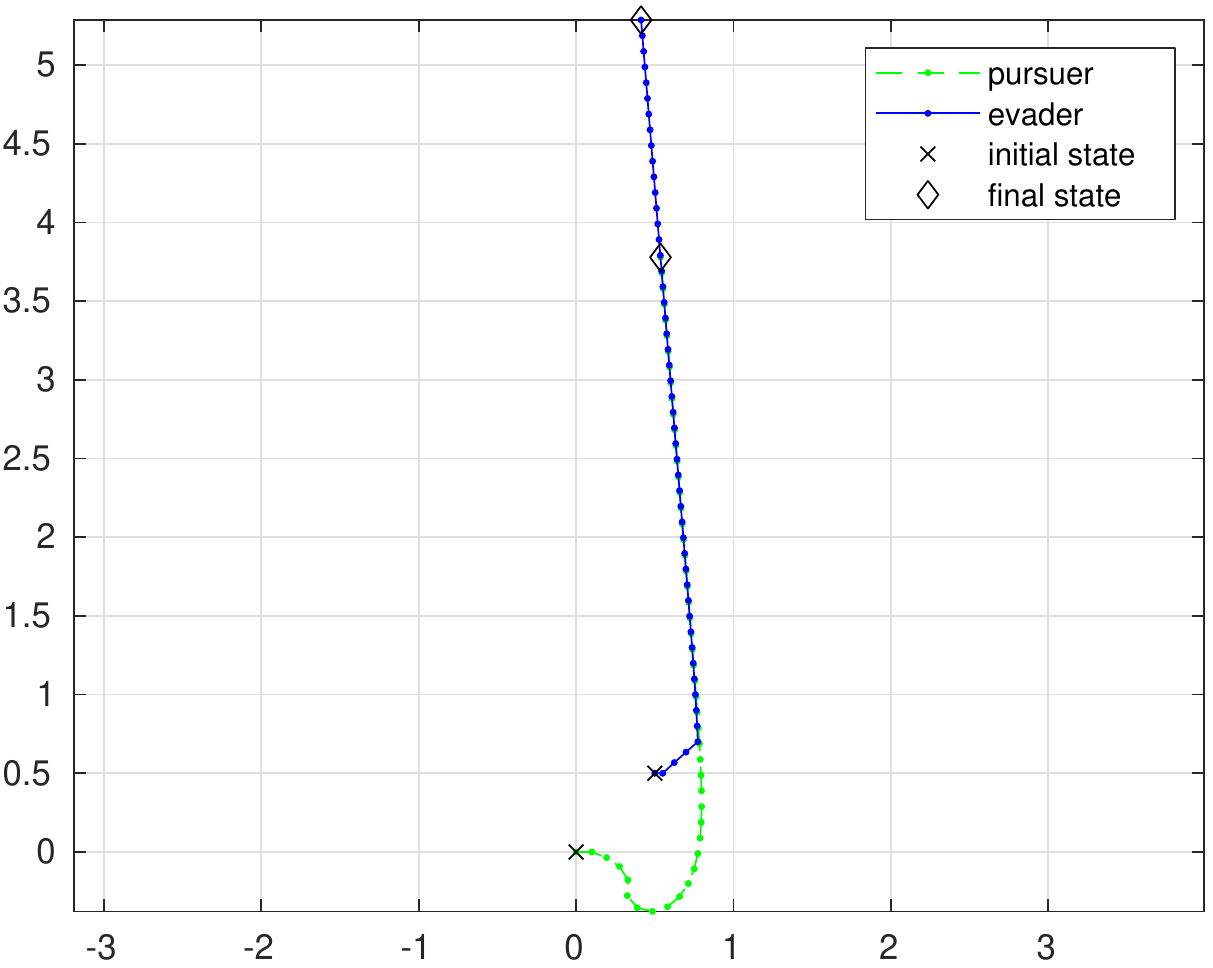}}
	\hfill  
	\subfloat[Without guaranteeing instability]{%
		\includegraphics[width=0.32\linewidth]{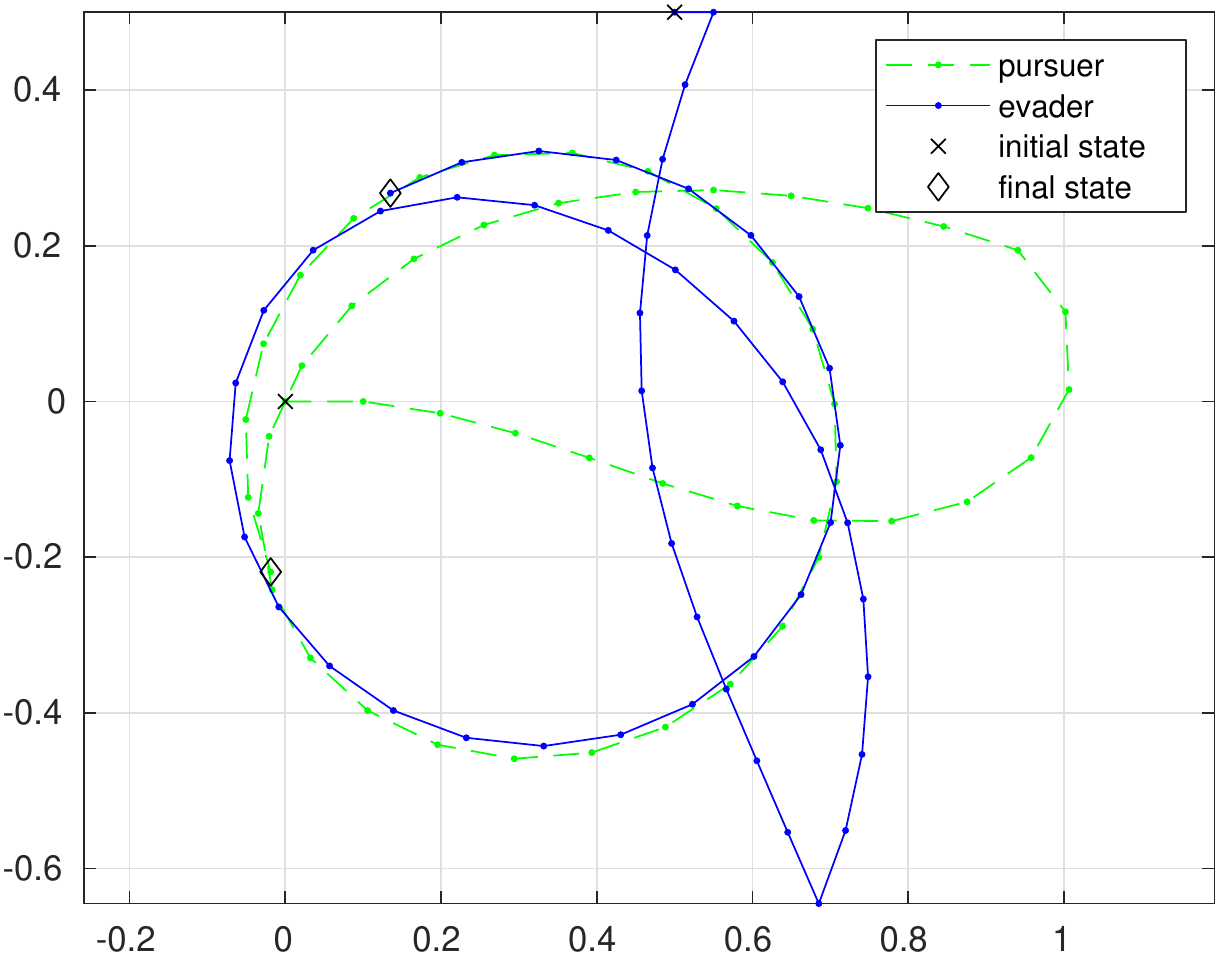}}
	\hfill  
	\subfloat[With guaranteeing instability only after $t=25$]{%
	\includegraphics[width=0.32\linewidth]{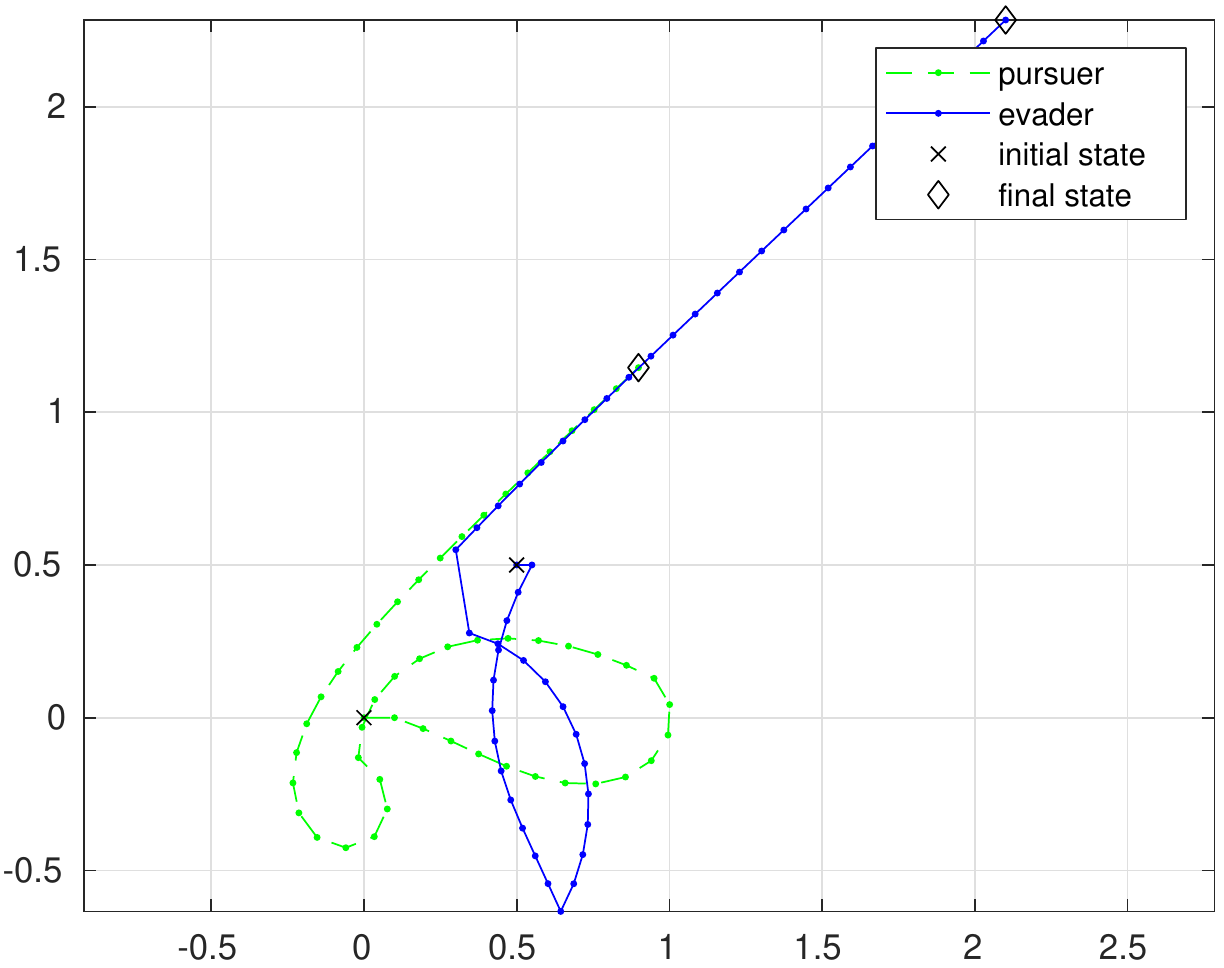}}	
	\caption{Trajectory for Homicidal Chauffeur problem with and without guaranteeing instability at equilibrium points that are not a local minmax.}
	\hfill 
	\label{fig:HC_stability} 
\end{figure*}

In the homicidal chauffeur problem, a pursuer driving a car is trying to hit a pedestrian, who (understandably) is trying to evade it. The pursuer is modeled as a discrete time Dubins's vehicle with equations
\begin{equation*}
x_p^+=\begin{bmatrix}
x_p^{(1)}+v \cos x_p^{(3)} \\
x_p^{(2)}+v \sin x_p^{(3)} \\
x_p^{(3)}+ u
\end{bmatrix}=:\phi_p(x_p,u)
\end{equation*}
where $x_p^{(i)}$ designates the i\textsuperscript{th} element of the vector $x_p$, $v$ is a constant forward speed and $u$ is the steering, over which the driver has control. The pedestrian is modeled by the accumulator
\begin{equation*}
x_e^+=x_e+d=:\phi_e(x_e,d)
\end{equation*}
where $d$ is the velocity vector. Given a time horizon $T$, and initial positions $x_e(t)$ and $x_p(t)$, we want to solve 
\begin{equation}\label{eq:mpc_pursuit_evader}
\min_{U \in \Ucal}
\quad\max_{D \in \Dcal}
\sum_{i=0}^{T-1} \norm{x_p^{(1,2)}(t+i+1)-x_e(t+i+1)}_2^2+\gamma_u u(t+i)^2-\gamma_d\norm{d(t+i)}_2^2
\end{equation}
where $x_p^{(1,2)}$ designates the first and second elements of the vector $x_p$; $\gamma_u$ and $\gamma_d$ are positive weights; and $U$, $\Ucal$, $D$ and $\Dcal$ are defined for $i=0,\dots,T-1$
\begin{align*}
&U:=u(t+i),x_p(t+i+1) 
\\&\Ucal:=\{u(t+i),x_p(t+i+1):\ u(t+i)^2\le u_{max}^2,  \\&\hspace{100pt} x_p(t+i+1)=\phi_p\big(x_p(t+i),u(t+i)\big)\}
\\&D:=d(t+i),x_e(t+i+1)
\\&\Dcal:=\{d(t+i),x_e(t+i+1):\ \norm{d(t+i)}_2^2\le d_{max}^2,
\\&\hspace{100pt} x_e(t+i+1)=\phi_e\big(x_e(t+i),d(t+i)\big)\}.
\end{align*}
Instead of explicitly computing the solution of the trajectory of the pursuer and evaders, we are implicitly computing them by setting the dynamics as equality constraints; we will show shortly that this has an important impact on the scalability of the algorithm.

Each player is controlled using Model Predictive Control (MPC), meaning that at each time step $t$ we solve \eqref{eq:mpc_pursuit_evader} obtaining controls $u(t)$ and $d(t)$, which are then used to control the system for the next time step. The problem satisfy the assumptions of Theorem \ref{th:stab-ip-minmax}, as it is differentiable and has local minmax points for which the LICQ and strict complementarity hold.

\paragraph{The importance of guaranteeing instability}

It is natural to ask whether it is important to enforce the instability guarantee, specially in the case where the \eqref{eq:conslqac} is not enough, meaning one needs to use line \ref{algline:instability} of Algorithm \ref{alg:ip-minmax}. 
In Figure \ref{fig:HC_stability} we show what can happen if they are not enforced. We take the homicidal chauffeur problem with a horizon of $T=20$ and we run the MPC control for $t=1,\dots,50$. In one case we enforce the instability guarantee, meaning that we use line \ref{algline:instability} of Algorithm \ref{alg:ip-minmax}, on the second case we only enforce the \eqref{eq:conslqac}, and on the third case we only enforce the instability guarantees after $t=25$. In all cases, we start the system with the exact same initial conditions.

In the first case, the evader (which is the maximizer), is able to find a control that allows it to get further from the pursuer. The average cost for all the time steps ($t=1,\dots,50$) ends up being around $0.2$. In the second case, the solver keeps being attracted towards a point that is not a local minmax (and more precisely, not a local maximum), which means that the evader is not capable of escaping the pursuer; as a consequence, the average cost for all the time steps ends up being around $0.05$, which is lower, as expected. Finally, in the third case, at $t=25$ the solver starts to be able to converge towards a local minmax, and the evader is able to escape from the pursuer.

This example illustrates how crucial it is to enforce instability. By doing it, we guarantee that the algorithm can only converge towards an equilibrium point that is a local minmax, and this can completely change the numerical solution. 

\paragraph{Exploiting sparsity}

\begin{figure*}[t!]
	\centering
	\subfloat[Computational scaling for solving homicidal chauffeur per horizon length]{%
		\includegraphics[width=0.46\linewidth]{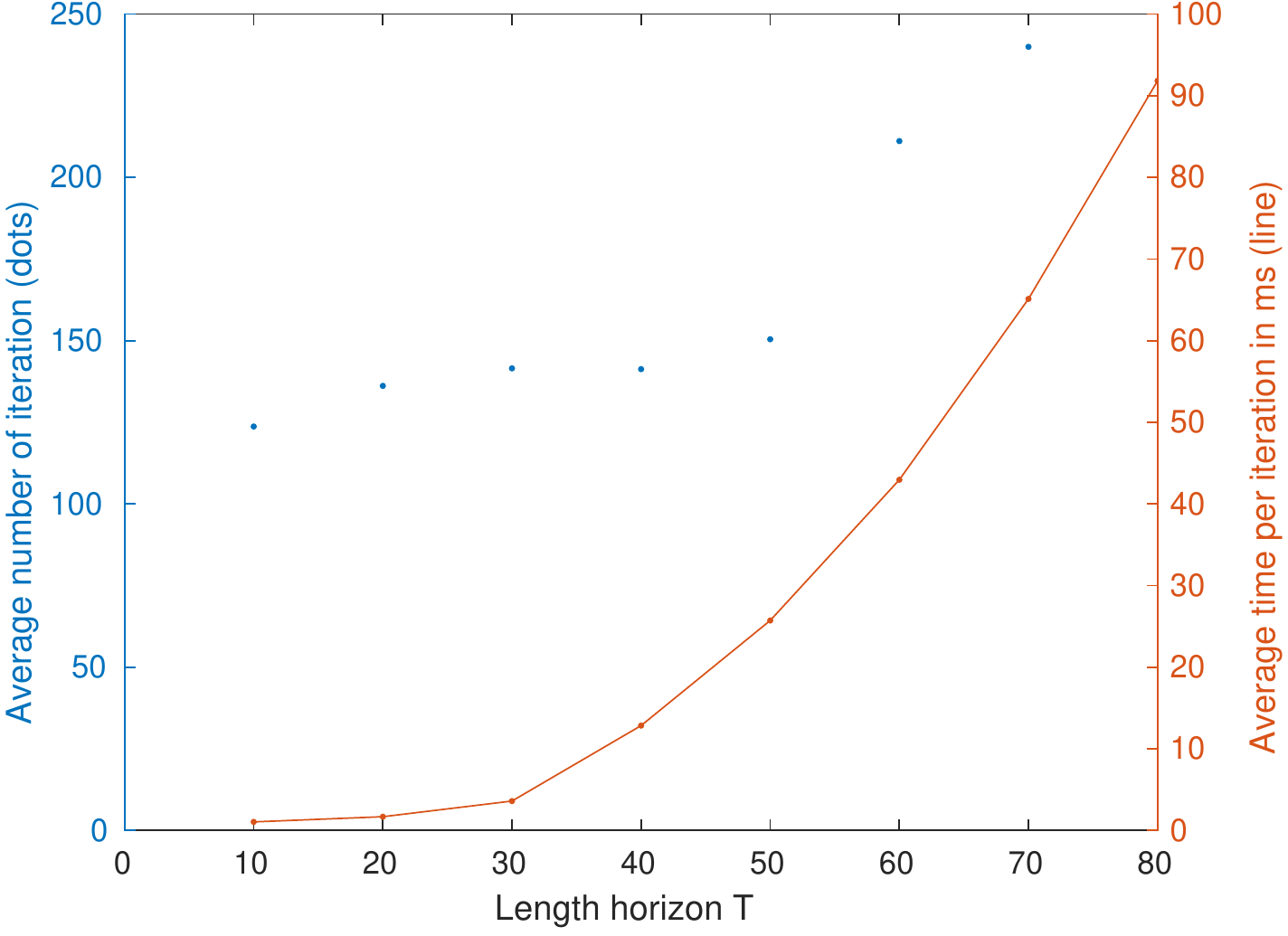}\label{subfig:scaling_noequal}}
	\hfill  
	\subfloat[Structural sparsity pattern of $J_{zz}f(z)$]{\raisebox{10pt}{
			\includegraphics[width=0.38\linewidth]{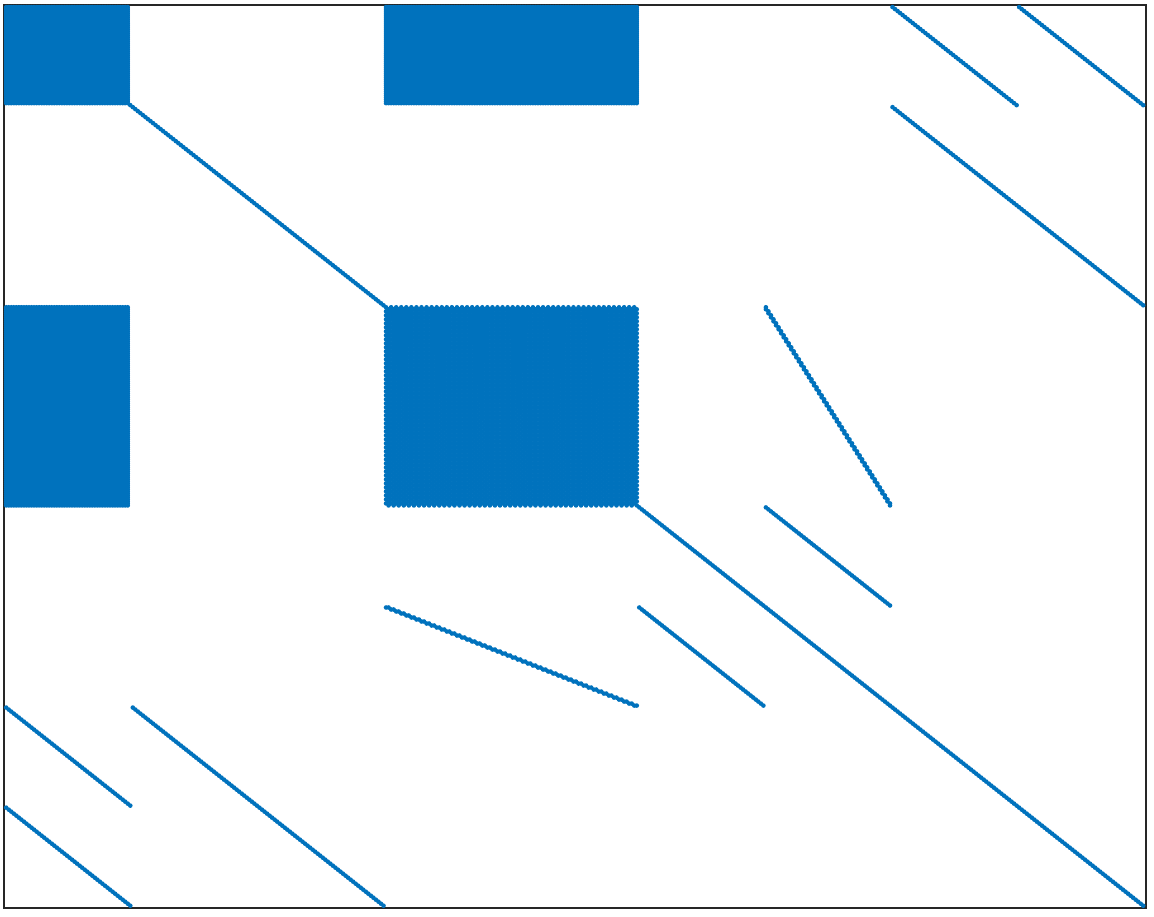}}\label{subfig:sparsity_noequal}}
	\caption{Scaling of homicidal chauffeur with horizon length and sparsity pattern of the Hessian when using the sequential approach}
	\label{fig:scale_sparsity_noequal}
\end{figure*}

Instead of setting the dynamics as equality constraints in \eqref{eq:mpc_pursuit_evader}, one could simply find the solution of the trajectory equation at each time step. This means to explicitly calculate $x_p(t+i+1)=\phi_p\Big(\phi_p\big(\dots,u(t+i-1)\big),u(t+i)\Big)$. In the MPC literature, this is known as the sequential approach, versus the simultaneous approach we used in \eqref{eq:mpc_pursuit_evader} \cite[Chapter 8.1.3]{rawlings_model_2017}. We want to study the scalability of the algorithm by enlarging the horizon $T$, both when using the sequential and the simultaneous approaches.

The sequential approach solves an optimization problem in a smaller state space, because it only needs to solve the optimization for $u(t),\dots,u(t+T)$ and $d(t),\dots,d(t+T)$ and it does not have to handle equality constraints. However, as we can see from the sparsity pattern in Figure \ref{subfig:sparsity_noequal}, the Hessian is rather dense, with large parts of it containing nonzero entries. As it can be seen in Figure \ref{subfig:scaling_noequal}, the algorithm scales rather poorly as the horizon length (and hence, the number of variables) increases; it no longer converges reliably after $T=80$. 

The simultaneous approach on the other hand solves the optimization problem in a much larger space state, because not only it needs to also solve for $u(t),\dots,u(t+T)$ and $d(t),\dots,d(t+T)$, but also for $x_p(t),\dots,x_p(t+T)$ and $x_e(t),\dots,x_e(t+T)$ and it also needs to handle equality constraints. Fortunately, as we can see from the sparsity pattern in Figure \ref{subfig:sparsity}, most of the entries in the Hessian are actually structurally zero (meaning they are always zero). \texttt{TensCalc}'s implementation of the LDLt factorization exploits sparsity patterns and scales roughly in $O(T)$, which makes it substantially more efficient than standard LDLt decomposition, which scales in $O(T^3)$ \cite[Appendix A]{nocedal_numerical_2006}. At each step of Algorithm \ref{alg:ip-minmax}, most of the time is spent computing the LDLt decomposition, either for adjusting $\epsilon_x$ and $\epsilon_y$ or to invert $H_{zz}f(z)$. As a consequence, we can see in Figure \ref{subfig:scaling} that both the number of iterations necessary to solve the optimization as well as the time per iteration scale roughly linear, the first being multiplied by about $1.7$ while the second by $3.5$ while the horizon length $T$ is multiplied by roughly $30$.

\begin{figure*}[t!]
	\centering
	\subfloat[Computational scaling for solving homicidal chauffeur per horizon length]{%
		\includegraphics[width=0.46\linewidth]{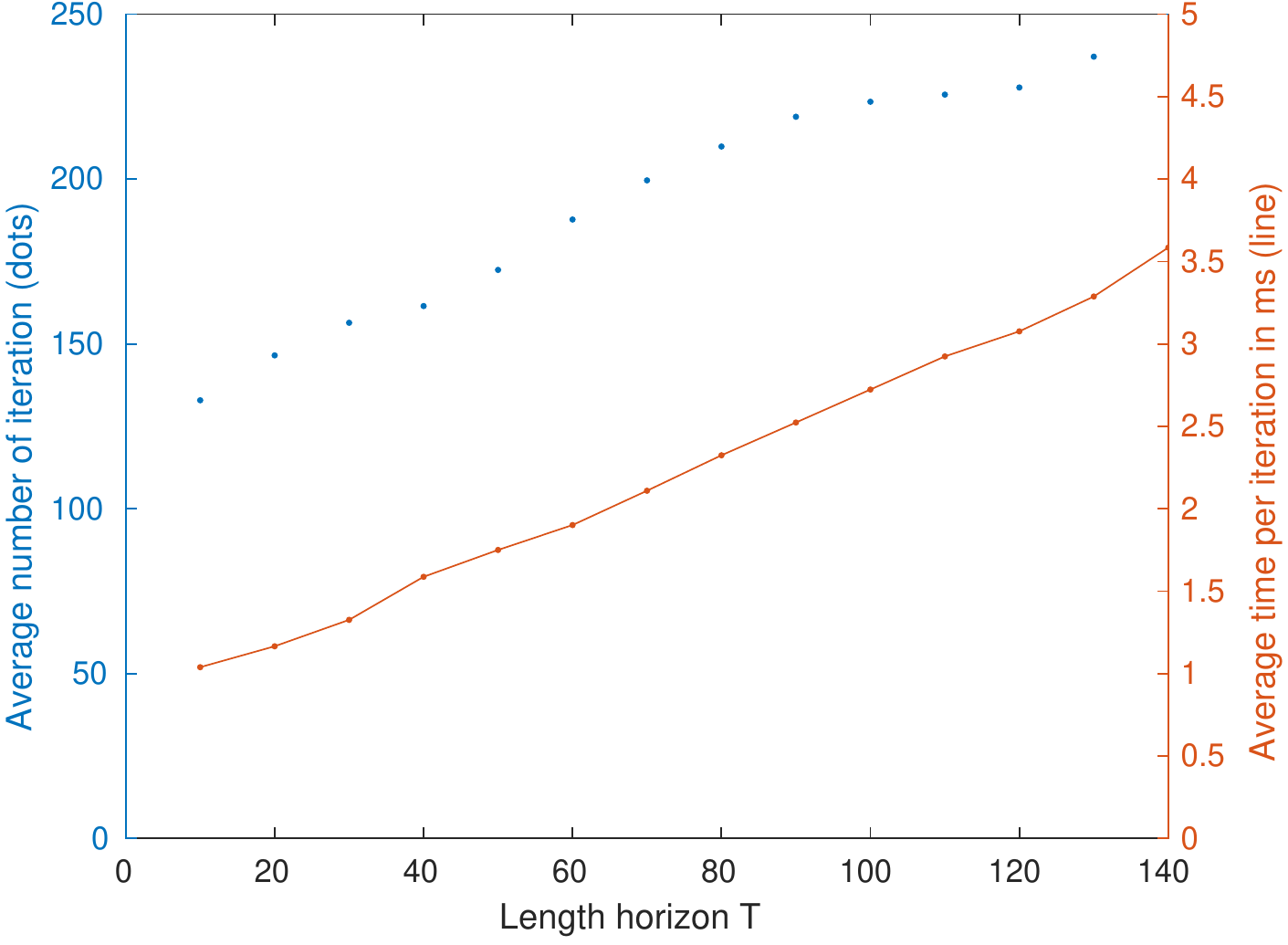}\label{subfig:scaling}}
	\hfill  
	\subfloat[Structural sparsity pattern of $J_{zz}f(z)$]{\raisebox{10pt}{
			\includegraphics[width=0.38\linewidth]{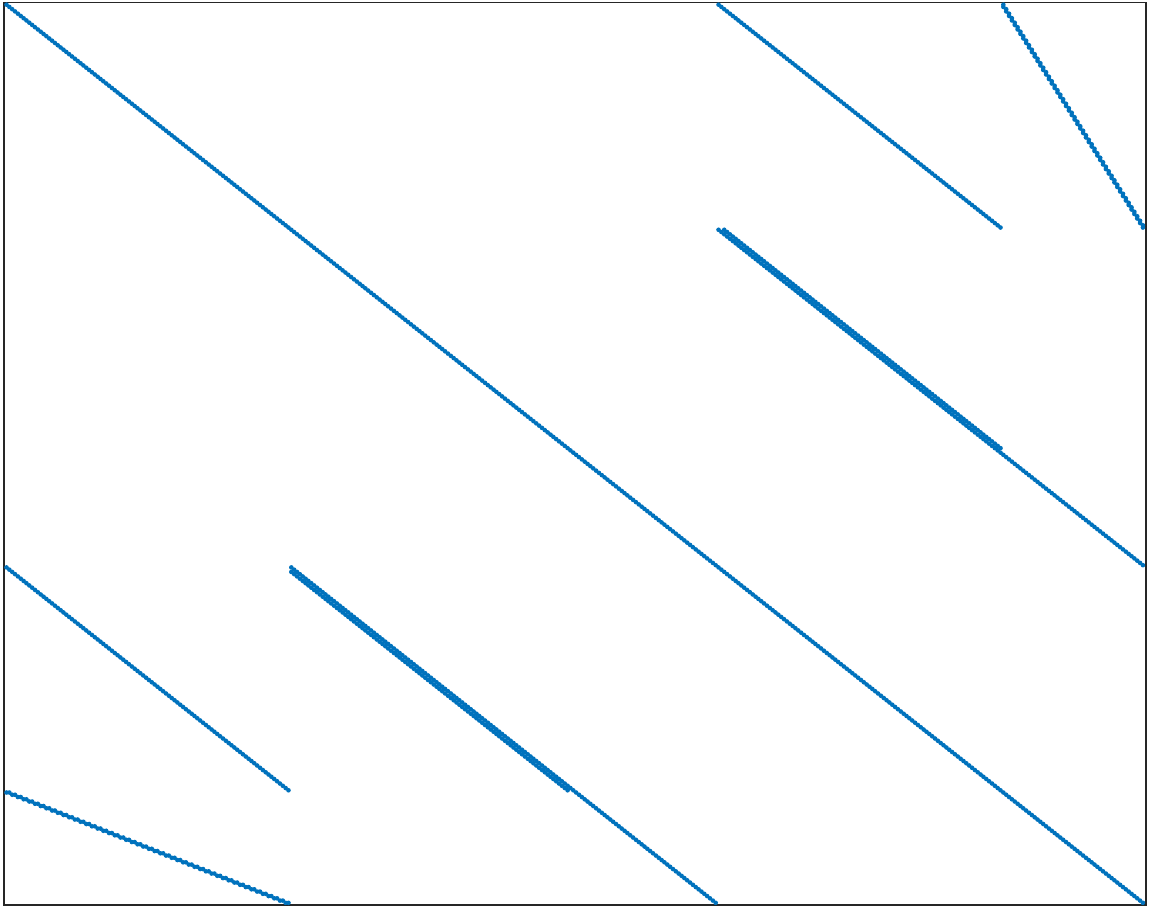}}\label{subfig:sparsity}}
	\caption{Scaling of homicidal chauffeur with horizon length and sparsity pattern of the Hessian}
	\label{fig:scale_sparsity}
\end{figure*}

\begin{remark}[Minmax problems with shared dynamics]
In the homicidal chauffeur, the control of the pursuer does not impact the \emph{dynamics} of the evader, and vice versa. This is why in \eqref{eq:mpc_pursuit_evader} the dynamics can be set as equality constraints independently for the min and for the max. 

Now consider the problem 
\begin{equation*}
x^+=f(x,u,d)
\end{equation*}
where $u$ is the control and $d$ is the disturbance and one wants to minimize a cost function $V(x(1),\dots,x(T),u(0),\dots,u(T-1))$ given the worst disturbance $d(1),\dots,d(T)$. Because both the control and the disturbances influence the dynamics, we need to include the dynamics as equality constraints for the \emph{maximization}, leading to the optimization problem
\begin{equation*}
\min_{u(i)\in \Ucal, i=0,\dots,T-1}\max_{\substack{d(i)\in\Dcal,x(i+1),i=0,\dots,T-1: \\ x(i+1)=f(x(i),u(i),d(i))}} V\Big(x(1),\dots,x(T),u(0),\dots,u(T-1)\Big)
\end{equation*}
where $\Ucal,\Dcal$ are the feasible sets for the control and disturbances. It is important to notice that $x$ just acts as a latent/dummy variable that allows us to avoid solving the trajectory equation. Setting it as a maximization variable does not changes the result as $x$ is always exactly determined by the value of $u$ and $d$. It does, however, improves the numerical efficiency of the algorithm as now the Hessian matrices are sparse and their LDL decomposition can be efficiently computed.\hfill $\square$%
\end{remark}

\section{Conclusion}

The main contribution of this article is the construction of Newton and primal-dual interior-point algorithm for nonconvex-nonconcave minmax optimization that can only converge towards an equilibrium point if such point is a local minmax. We established this results by modifying the Hessian matrices such that the update steps can be seen as the solution of quadratic programs that locally approximate the minmax problem. While our results are only local, using numerical simulations we see that the algorithm is able to make progress towards a solution even if it does not start close to it. We also illustrated using numerical examples how important it is to have a formulation of the minmax problem such that the Hessian matrix is sparse. 

The main future direction would be to develop non-local convergence results. We believe that the best approach to obtain such results would be to develop a type of Armijo rule which could be used to obtain similar results to those from minimization. Developing filters and merit function could also play an important role in coming up with ways to improve the algorithm's convergence. 

\section{Acknowledgment}
The authors would like to thank the reviewers and editors for their comments which greatly improved the quality of the article and helped to clarify the main challenges we were addressing. The first author, Raphael Chinchilla, would also like to acknowledge Prof. Stephen Boyd for his suggestions and interest in the paper.

\subsection{Funding}

This material is based upon work supported by the U.S. Office of Naval Research under the MURI grant No. N00014-16-1-2710.

\appendix
\begin{appendices}

\section{Second order sufficient conditions for constrained minimization and minmax optimization} \label{app:proof_propo_secon}

\subsection{Proof of Proposition \ref{prop:constrsecondorder_min} (constrained minimization)}

The first step is to show that $g(z,0)=\zerob$ is equivalent to the  Karush--Kuhn--Tucker (KKT) conditions \cite[Chapter 12]{nocedal_numerical_2006}. Consider the  ``full" Lagrangian $\tilde L(x,s_x,\nu_x,\lambda_x,\tau_x)=f(x)+\nu_x'G_x(x)+\lambda_x'(F_x(x)+s_x)-\tau_x's_x$ for the optimization \eqref{eq:minimization_slack}. The KKT condition would then be that 
\begin{equation}\label{eq:kkt_naive_slack}
\begin{bmatrix}
\grad_x \tilde L(x,s_x,\nu_x,\lambda_x,\tau_x) \\
\grad_{s_x} \tilde L(x,s_x,\nu_x,\lambda_x,\tau_x)= \lambda_x-\tau_x \\
G_x(x)\\
F_x(x)+s_x\\
\tau_x \odot s_x
\end{bmatrix}=\zerob 
\end{equation}
and $s_x,\tau_x \ge\zerob$. The second equation can be used to substitute $\tau_x$ by $\lambda_x$, which gives the equality $g(z,0)=\zerob$.

Now the second order sufficient conditions. Let us start by rewriting the minimization \eqref{eq:minimization} but instead of using as slack variables $s_x$ with the constraint $s_x\ge0$, using the slack variable $w_x\odot w_x$ (where $\odot$ is the element wise product):
\begin{equation}\label{eq:minimization_slack_w}
\min_{x,w_x:G_x(x)=\zerob,F_x(x)+w_x\odot w_x=\zerob} f(x).
\end{equation}
Consider now the solution cone
\begin{multline*}
\Ccal_x(z):=\{(d_x,d_w) \in \eR^{n_x+ m_x}\backslash\{\zerob \}: \grad_x G_x(x)'d_x=\zerob,\\ \grad_x F_x(x)d_x + 2\diag(w_x)d_w=\zerob\}
\end{multline*}
Let $(x,w_x,\nu_x,\lambda_x)$ be a point such that the KKT conditions for \eqref{eq:minimization_slack_w} hold. As, by assumption, the LICQ and strict complementarity conditions hold, if
\begin{equation}\label{eq:second_order_min_cone}
\begin{bmatrix}
d_x\\
d_w
\end{bmatrix}'
\begin{bmatrix}
\grad_{xx} L(z) & \zerob \\
\zerob & 2\diag(\lambda_x) \\
\end{bmatrix}
\begin{bmatrix}
d_x\\
d_w
\end{bmatrix}'>0 \ \forall (d_x,d_w)\in \Ccal_x(z)
\end{equation}
then $(x,w_x,\nu_x,\lambda_x)$ is a local minimum of \eqref{eq:minimization_slack_w}. The proof can be found in \cite[Theorem 12.5]{nocedal_numerical_2006}. 

We now need to prove that \eqref{eq:second_order_min_cone} is equivalent to the condition \eqref{eq:second_order_inertia_constrained_min} from the proposition. Because the LICQ and strict complementarity hold, the set $\Ccal_x(z)$ is given by the null space (a.k.a. the kernel) of the matrix
\begin{equation}
\tilde H_{x\lambda}f(z)=\begin{bmatrix}
\grad_{x} G_x(x) & \grad_{x}F_x(x) \\
0 & 2\diag(w_x).
\end{bmatrix}
\end{equation}
This result can be found in \cite[Chapter 12.5]{nocedal_numerical_2006}, in the subsection ``Second-order conditions and projected Hessian". Let $Z_x\in\eR^{n_x+m_x,n_x+m_x-m_x-l_x}$ be a matrix with full column rank such that $ \tilde  H_{x\lambda}f(z)'\,Z_x=\zerob$. Then, the condition \eqref{eq:second_order_min_cone} can be rewritten as
\begin{equation*}
Z_x'\begin{bmatrix}
\grad_{xx} L(z) & \zerob \\
\zerob & 2\diag(\lambda_x) \\
\end{bmatrix}Z_x \succ 0
\end{equation*}
which is equivalent to say that 
$$
\inertia\qty(Z_x'\begin{bmatrix}
\grad_{xx} L(z) & \zerob \\
\zerob & 2\diag(\lambda_x) \\
\end{bmatrix}Z_x)=(n_x-l_x,0,0)
$$

Now consider the matrix
\begin{equation}
\tilde H_{zz}f(z)=
\begin{bmatrix}
\grad_{xx} L(z) & \zerob & \grad_{x} G_x(x) & \grad_{x}F_x(x) \\
\zerob & 2\diag(\lambda_x) & \zerob & 2\diag(w) \\
\grad_{x} G_x(x)'& \zerob & \zerob & \zerob\\
\grad_{x}F_x(x)'& 2\diag(w) &\zerob & \zerob  \\
\end{bmatrix}.
\end{equation}
As the LICQ conditions hold, according to \cite[Theorem 16.3]{nocedal_numerical_2006}
$$
\inertia( \tilde H_{zz}f(z))=\inertia\qty(Z_x'\begin{bmatrix}
\grad_{xx} L(z) & \zerob \\
\zerob & 2\diag(\lambda_x) \\
\end{bmatrix}Z_x )+(l_x+m_x,l_x+m_x,0).
$$
Therefore \eqref{eq:second_order_min_cone} holds if and only if $\inertia(\tilde  H_{zz}f)=(n_x+m_x,l_x+m_x,0)$. 

We have almost finished the proof, we now just need to prove that $\inertia(\tilde H_{zz}f(z))=\inertia(H_{zz}f(z))$. Using the equality condition $F_x(x)+w_x\odot w_x=0$, we obtain the relation $w_x=(-F_x(x))^{1/2}=s_x^{1/2}$. If we substitute back this result in $\tilde H_{zz}f(z)$ we almost have that $\tilde H_{zz}f(z)$ is equal to $H_{zz}f(z)$ except for the $2$ in front of $\diag(\lambda_x)$ and $\diag(s^{1/2})$. Take the matrix $\Xi$ defined by
$$
\Xi=\diag([\oneb_{n_x},[a^{(1)},a^{(2)},\dots,a^{(m_x)}],\oneb_{l_x+m_x}])
$$
where
$$
a^{(i)}=
\begin{cases}
\frac{1}{2} & \tif \lambda_x^{(i)}=0 \tand s_x^{(i)}\ne0  \\
\frac{1}{\sqrt{2}} & \tif \lambda_x^{(i)}\ne0 \tand s_x^{(i)}=0
\end{cases}
$$
with $\lambda_x^{(i)} \tand s_x^{(i)}$ denoting the $i^\text{th}$ elements of $\lambda_x \tand s_x$. Then $\Xi \tilde  H_{zz}f(z)\Xi =  H_{zz}f(z)$ which, according to Sylvester's law of inertia \cite[Theorem 1.5]{zhang_schur_2005}, implies that $\text{inertia}(\tilde H_{zz}f(z))=\text{inertia}(H_{zz}f(z))$, which finishes the proof. \hfill$\square$ 

\subsection{Proof of Proposition \ref{prop:second-order-constrained-minmax} (constrained minmax optimization)}
First, using the exact same reasoning as in the proof of Proposition \ref{prop:constrsecondorder_min}, one can show that $g(z,0)=0$ is equivalent to the first order necessary condition in \cite{dai_optimality_2020}.

Similarly to what we did in the proof of Proposition \ref{prop:constrsecondorder_min}, let us start by rewriting the constrained minmax optimization \eqref{eq:minmaxconstrained_slack} using the slack variables $w\odot w$:
$$
\min_{x,w_x: G_x(x)=\zerob,F_x(x)+w_x\odot w_x=\zerob}\quad \max_{y,w_y: G_y(x,y)=\zerob,F_y(x,y)+w_y\odot w_y=\zerob} f(x,y).
$$
Consider the solution cones
\begin{multline*}
\Ccal_y(z):=\{(d_y,d_{w_y}) \in \eR^{n_y + m_y}\backslash\{\zerob \}: \grad_y G_y(x,y)d_y=\zerob,\\ \grad_y F(x,y)d_y +2\diag(w_y)d_{w_y} =\zerob \}
\end{multline*}
and 
\begin{multline*}
\Ccal_x(z):=\{(d_x,d_{w_x}) \in \eR^{n_x + m_x}\backslash\{\zerob\}: \grad_x G_x(x)'d_x=\zerob\\ \grad_x F_x(x)d_x + 2\diag(w_x)d_{w_x}=\zerob\}
\end{multline*}

Let $z$ be a point such that $g(z,0)=\zerob$. As, by assumption, the LICQ and strict complementarity hold, if
\begin{subequations}
	\begin{equation} \label{eq:suffconstrmax}
	\begin{bmatrix}
	d_y\\
	d_{w_y}
	\end{bmatrix}'
	\begin{bmatrix}
	\grad_{yy} L(z) & \zerob \\
	\zerob & -2\diag(\lambda_y)
	\end{bmatrix}
	\begin{bmatrix}
	d_y\\
	d_{w_y}	
	\end{bmatrix}	
	<0\ \forall\ (d_y,d_{w_y}) \in \Ccal_y(z)
	\end{equation}
	and 
	\begin{equation} \label{eq:suffconstrmin}
	\begin{bmatrix}
	d_x\\
	d_{w_x}\!
	\end{bmatrix}'	
	\Big(H_{xx}L(z) -  H_{xy}f(z)  H_{yy}f(z)\inv  H_{xy}f(z)' \Big) 	\begin{bmatrix}
	d_x\\
	d_{w_x}
	\end{bmatrix}\!>\!0\ \forall\ (d_x,d_{w_x}) \in \Ccal_x(z)
	\end{equation}
\end{subequations}
then $(x,w_x,\nu_x,\lambda_x)$ is a local minimum of \eqref{eq:minimization_slack_w}. The proof can be found in  \cite[Theorem 3.2]{dai_optimality_2020}.

The proof between the equivalence of the condition \eqref{eq:suffconstrmax} and  $\inertia( H_{yy}f(z))=(l_y+m_y,n_y+m_y,0)$ is almost identical to the proof of Proposition \eqref{prop:constrsecondorder_min}. 

The condition on the inertia of $\inertia( H_{zz}f(z))$ require some more development. In an analogous way to the proof of Proposition \eqref{prop:constrsecondorder_min}, let $Z_x$ be a matrix with full column rank such that $H_{x\lambda}f(z)'\,Z_x=\zerob$. Then the sufficient conditions \eqref{eq:suffconstrmin} for the reformulated outer minimization is
\begin{equation} \label{eq:constminnullspace}
Z_x'\Big( H_{xx}f(z) -  H_{yx}f(z)'  H_{yy}f(z)\inv  H_{yx}f(z) \Big)Z_x \succ 0.
\end{equation}
We want now to define a new partition of $H_{zz}f(z)$ which we will use to finish the proof. Consider the matrices
$$
\bar H_{zz}f(z)= \begin{bmatrix}
H_{xx}f(z) &  H_{xy}f(z) \\
H_{xy}f(z)' &  H_{yy}f(z)
\end{bmatrix}
\quad \tand \quad
\bar H_{x\lambda}f(z) =
\begin{bmatrix}
H_{x\lambda}f(z) \\
\zerob_{n_y+m_y+l_y+m_y,l_x+m_x}
\end{bmatrix}.
$$
such that
$$
H_{zz}f(z)=
\begin{bmatrix}
\bar H_{zz}f(z) & \bar H_{x\lambda}f(z)\\
\bar H_{x\lambda}f(z) ' & \zerob_{l_x+m_x}
\end{bmatrix}
$$
Let the matrix
$$
\bar Z_x:=
\begin{bmatrix}
Z_x & \zerob_{n_x+m_x,n_y+m_y+l_y+m_y}\\
\zerob_{n_y+m_y+l_y+m_y,n_x-l_x} & I_{n_y+m_y+l_y+m_y}.
\end{bmatrix}
$$
One can show that $\bar Z_x$ is full column rank and such that $\bar H_{x\lambda}f(z)'\,\bar Z_x=\zerob$. Therefore if we apply \cite[Theorem 16.3]{nocedal_numerical_2006} to $ H_{zz}f(z)$ (with the new partitioning) gives 
\begin{equation*}
\inertia( H_{zz}f(z))=\inertia\qty(\bar Z_x' \bar H_{zz} f(z)\bar Z_x)+(l_x+m_x,l_x+m_x,0)
\end{equation*}

In turn, $\inertia\qty(\bar Z_x'\bar H_{zz} f(z)\bar Z_x)$ can be simplified using Haynsworth inertia additivity formula \cite[Theorem 1.6]{zhang_schur_2005}:
\begin{align*}
&\inertia\qty(\bar Z_x'\bar H_{zz} f(z)\bar Z_x)
\\&=\inertia\qty(\begin{bmatrix}
  Z_x'  H_{xx}f(z)Z_x & Z_x' H_{xy}f(z) \\
   H_{xy}f(z)'Z_x &  H_{yy}f(z)
\end{bmatrix})
\\&=\inertia\qty( Z_x'\Big( H_{xx}f(z) -  H_{xy}f(z)  H_{yy}f(z)\inv  H_{xy}f(z)' \Big)Z_x)+\inertia( H_{yy}f(z)).
\end{align*}
Therefore, if \eqref{eq:suffconstrmax} holds,  \eqref{eq:suffconstrmin} is equivalent to
$$
\inertia( H_{zz}f(z))= (n_x-l_x,0,0)+(l_y+m_y,n_y+m_y,0)+(l_x+m_x,l_x+m_x,0)
$$
which finishes the proof. \hfill$\square$ 

\end{appendices}

\bibliography{biblio}

\end{document}